\documentclass[dvipsnames,12pt]{amsart}
\usepackage[margin=1.3in]{geometry}
\usepackage[T1]{fontenc}
\usepackage[utf8]{inputenc}
\usepackage[dvipsnames]{xcolor}
\usepackage{graphicx}
\usepackage{tikz-cd}
\usepackage{stmaryrd}
\usepackage[pagebackref]{hyperref}
\hypersetup{
  colorlinks=true,
    linkcolor=purple,
    filecolor=blue,
    urlcolor=blue,
    citecolor=ForestGreen,
    linktoc=page}

\SetSymbolFont{stmry}{bold}{U}{stmry}{m}{n}
\usepackage[scr=rsfs,bb=pazo]{mathalfa}
\usepackage[abbrev]{amsrefs}
\usepackage{microtype,amsmath,amsfonts,amssymb,amsthm,mathrsfs,bm}

\theoremstyle{plain}
\newtheorem{theorem}[]{Theorem}
\newtheorem{theorema}{Theorem}

\newtheorem{proposition}[theorem]{Proposition}
\newtheorem{lemma}[theorem]{Lemma}

\newtheorem{corollary}[theorem]{Corollary}

\theoremstyle{definition}
\newtheorem{definition}[theorem]{Definition}
\newtheorem{instance}[theorem]{Example}
\newtheorem{situation}[theorem]{}

\usepackage{enumitem}
\numberwithin{theorem}{section}
\numberwithin{equation}{section}
\setlist[itemize]{leftmargin=2.5em}

\title[Twisted GKZ and relative cohomology]{Twisted GKZ hypergeometric functions and relative homology}
\date{\today}

\author[Tsung-Ju~Lee]{Tsung-Ju Lee}
\email{tjlee@cmsa.fas.harvard.edu}
\address{Center of Mathematical Sciences and Applications, 20 Garden St., Cambridge, MA 02138.}

\author[Dingxin~Zhang]{Dingxin Zhang}
\email{dingxin@tsinghua.edu.cn}
\address{Yau Mathematical Sciences Center, Tsinghua University, Beijing 100084, China.}

\begin{document}
\begin{abstract}
We investigate the GKZ \(A\)-hypergeometric \(\mathscr{D}\)-modules,
introduced by Gel'fand, Kapranov, and Zelevinskii, arising from cyclic covers
of toric varieties and find its Riemann--Hilbert partner.
This extends our earlier results in
\cite{2020-Lee-Zhang-a-hypergeometric-systems-and-relative-cohomology}.
\end{abstract}
\maketitle
\tableofcontents
\section{Introduction}
A GKZ \(A\)-hypergeometric system (or a GKZ \(A\)-hypergeometric
\(\mathscr{D}\)-module), introduced by
Gel'fand, Graev, Kapranov, and Zelevinskii \cites{1987-Gelfand-Graev-Zelevinski-holonomic-systems-of-equations-and-series-of-hypergeometric-type,1989-Gelfand-Kapranov-Zelevinski-hypergeometric-functions-and-toral-manifolds}, is a
system of linear partial differential equations generalizing
the hypergeometric structure which can be traced back to Euler and Gauss.
The inputs of the system are an integral matrix \(A\in\mathrm{Mat}_{d\times m}(\mathbb{Z})\)
together with a parameter \(\beta\in\mathbb{C}^{d}\) and the output
is a system of partial differential equations on \(\mathbb{C}^{m}\).
GKZ \(A\)-hypergeometric \(\mathscr{D}\)-modules appear in various branches
of mathematics and its solution has found a variety of applications in number theory,
algebraic geometry as well as mirror symmetry.

The recent work of Hosono, Lian, Takagi, and Yau
\cites{2020-Hosono-Lian-Takagi-Yau-k3-surfaces-from-configurations-of-six-lines-in-p2-and-mirror-symmetry-i,2019-Hosono-Lian-Yau-k3-surfaces-from-configurations-of-six-lines-in-p2-and-mirror-symmetry-ii-lambda-k3-functions} shed light on mirror symmetry for
singular Calabi--Yau varieties and drew our attention to
periods for \emph{cyclic covers} of toric varieties.
It can be checked that the periods for equisingular
families of cyclic covers of toric varieties
are also governed by a certain
type of GKZ \(A\)-hypergeometric system.
One of the most important features is
that the parameter \(\beta\) is no longer an integral vector. Instead,
it is a \(\mathbb{Q}\)-vector.
This distinguishes cyclic covers from classical complete intersections
in toric varieties.
The purpose of this paper is to give a
cohomological description
of the solution space to such a GKZ
\(A\)-hypergeometric \(\mathscr{D}\)-module
under Riemann--Hilbert correspondence.

Our main result can be applied to more general GKZ
\(A\)-hypergeometric systems, not limited to the ones from
cyclic covers of toric manifolds.
Nonetheless, to give a more concise statement, we
will state our results in a slightly restricted form in a special situation.

Let \(X\) be a projective smooth toric variety over \(\mathbb{C}\) and
\(\mathcal{L}^{-1}\) be
a big and numerically effective line bundle over \(X\).
Any \(s\in\mathrm{H}^{0}(X,\mathcal{L}^{-k})\)
gives rise to an \(k\)-fold cyclic cover of \(X\) branched
over \(\{s=0\}\) by the fibred square
\begin{equation*}
\begin{tikzcd}
&Y\ar[r]\ar[d] & \mathbb{L}^{-1}\ar[d]\\
&X\ar[r,"s"] &\mathbb{L}^{-k}.
\end{tikzcd}
\end{equation*}
Let \((t_{1},\ldots,t_{n})\) be the coordinate on the maximal torus of \(X\).
Denote by \(\{t^{w_{1}},\ldots,t^{w_{m}}\}\) the integral points
in the divisor polytope of \(\mathcal{L}^{-k}\).
We further assume that \(\mathbf{0}\) is an interior
point in the divisor polytope.
The universal section \(\sigma=\sum_{i=1}^{m} x_{i} t^{w_{i}}\)
gives rise to the universal family of cyclic covers
\(\mathcal{Y}\to \mathrm{H}^{0}(X,\mathcal{L}^{-k})\) whose
``period integrals'' on the maximal torus in \(X\) are of the form
\begin{equation*}
\int \sigma^{1/k-1}\frac{\mathrm{d}t_{1}}{t_{1}}\wedge\cdots\wedge
\frac{\mathrm{d}t_{n}}{t_{n}}.
\end{equation*}

The appearance shows that they are governed by a GKZ \(A\)-hypergeometric
\(\mathscr{D}_{\mathbb{C}^{m}}\)-module \(\mathcal{M}_{A}^{\beta}\) in variables \(x_{i}\).
Here \(\mathscr{D}_{\mathbb{C}^{m}}=\mathbb{C}[x_{i},\partial_{x_{j}}]\) is
the Weyl algebra on \(\mathbb{C}^{m}\).
Under this circumstance, our main result can be stated as follows.
\begin{theorema}
\label{thm:main-theorem-intro}
For \(b\in \mathrm{H}^{0}(X,\mathcal{L}^{-k})\), we have
\begin{equation*}
\operatorname{Sol}^{0}(\mathcal{M}_{A,\beta})_{b}
\cong \mathrm{H}_{n}(U_{b},U_{b}\cap D_{\infty},\mathscr{L}_{\beta,b}).
\end{equation*}
Here \(U_{b}=X\setminus \{b=0\}\), \(D_{\infty}\) is the union
of all toric divisors, \(\mathscr{L}_{\beta,b}\)
is the local system on \(U_{b}\) whose monodromy exponent around \(\{b=0\}\)
is \(1/k-1\), and \(\operatorname{Sol}^{0}(-)=
R^{0}\mathcal{H}om_{\mathscr{D}_{\mathbb{C}^{m}}^{\mathrm{an}}}
((-)^{\mathrm{an}},\mathscr{O}_{\mathbb{C}^{m}}^{\mathrm{an}})\) is the
classical solution functor.
\end{theorema}

As we have mentioned, our result (cf.~Theorem \ref{thm:main}) is more general; we may
allow \(\mathcal{L}^{-1}\) to be
a split vector bundle (i.e.~a product of line bundles) and the
exponent \(\beta\)
can be an arbitrary \(\mathbb{Q}\)-vector as long as it remains
\emph{semi-nonresonant} in the sense of Mutsumi~Saito
\cite{2001-Saito-isomorphic-classes-of-a-hypergeometric-systems}.

We remark that
the Riemann--Hilbert problem of GKZ \(A\)-hypergeometric
systems was settled by Gel'fand \emph{et al}.~when
\(\beta\) is \emph{non-resonant}
\cite{1990-Gelfand-Kapranov-Zelevinsky-generalized-euler-integrals-and-a-hypergeometric-functions},
which has been used by the first author to prove
the completeness of the GKZ \(A\)-hypergeometric system
arising from periods for Calabi--Yau double covers of
toric manifolds
\cite{2022-Lee-a-note-on-periods-of-calabi-yau-fractional-complete-intersections}.
Besides cyclic covers of toric manifolds, one can also consider
cyclic covers of homogeneous spaces and
use tautological systems to study their periods.
Recall that tautological \(\mathscr{D}\)-modules
are introduced by Lian \emph{et al}.~\cite{2013-Lian-Song-Yau-periodic-integrals-and-tautological-systems}
to tackle the periods for Calabi--Yau
hypersurfaces or complete intersections
in homogeneous spaces.
In the case of cyclic covers of homogeneous spaces, one should replace
the GKZ \(A\)-hypergeometric system by
a tautological system with a fractional exponent.
This has been studied by G\"{o}rlach \emph{et al}.~in their recent work
\cite{2022-Gorlach-Reichelt-Sevenheck-Steiner-Walther-tautological-systems-homogeneous-spaces-and-the-holonomic-rank-problem}.

The proof of Theorem \ref{thm:main} relies heavily on the
result of Reichelt
\cite{2014-Reichelt-laurent-polynomials-gkz-hypergeometric-systems-and-mixed-hodge-modules}.
For non-semiresonant \(\beta\),
the GKZ \(A\)-hypergeometric \(\mathscr{D}\)-module \(\mathcal{M}_{A}^{\beta}\)
can be identified with a complex of holonomic \(\mathscr{D}\)-modules under
Fourier--Laplace transform, i.e.~\(\mathcal{M}_{A}^{\beta}=\mathrm{FT}(\mathcal{N})\).
The remaining task is to
compute \(\mathrm{FT}(\mathcal{N})\) explicitly and give it a topological meaning.
  The computations are performed in Sections~\ref{sec:regularity} and \ref{sec:push},
  where we exhibit its regularity and relate it to toric geometry.
  It is worth noting that the Fourier--Laplace transform is closely related
  to exponentially twisted \(\mathscr{D}\)-modules, which are generally irregular.
  It is the homogeneity condition of the GKZ \(A\)-hypergeometric system that ensures the regularity of the
  final output.

  Exponential twists of integrable connections are an
  algebraic recipe for computing vanishing cycles in the theory of
  \(\mathscr{D}\)-modules.  The main idea is that the exponentially
  twisted cohomology should account for the ``shapes'' of the critical
  points of the function (``stationary phase approximation'').  
  The structure of critical values of the function that appear in the 
  twisted GKZ system is simple, so we may a direct computation 
  in Section~\ref{sec:regularity}.

\subsection*{Acknowledgement}
Part of the results in this paper was presented in the conference
entitled \emph{Calabi--Yau manifolds and mirror symmetry - Past, Present, and Future -}
held in Gakushiun University in Japan
in August 2022. We thank the organizers for the invitation.
Tsung-Ju Lee is partially supported by the Simons Collaboration Grant on Homological
Mirror Symmetry and Applications 2015--2022.
Dingxin Zhang is partially supported by the national key
research and development program of China (No.~2022YFA1007100).

\section{GKZ \texorpdfstring{\(A\)}{A}-hypergeometric systems}
\label{sec:gkz-system}
In this section, we recall the
definition of GKZ \(A\)-hypergeometric
systems and
give the precise statement of our main theorem.
\begin{situation}
\label{situation:a-hypergeometric-notation}
We begin with the definition of GKZ
\(A\)-hypergeometric systems. Fix a positive integer \(r\ge 1\).
\begin{enumerate}
  \item Let \(V_{i} = \mathbb{C}^{m_{i}}\) be a complex vector space of dimension
    \(m_{i}\) for each \(1\le i\le r\). Put \(m=m_{1}+\cdots+m_{r}\)
    and \(V=V_{1}\times\cdots\times V_{r}\).
  \item Let \(x_{i,1},\ldots, x_{i,m_{i}}\) be a fixed coordinate system on the
    \emph{dual} vector space \({V}_{i}^{\vee}\). We put
    \(\partial_{i,j}=\partial/\partial x_{i,j}\).
  \item For each \(1\le k\le r\), let \(A_{k}\) be an integral matrix of the form
    \begin{equation*}
      A_{k}= (a^{k}_{i,j}) =
	\begin{bmatrix}
	e_{k} & \cdots & e_{k}\\
	\vline height 1ex & & \vline height 1ex\\
	w_{k,1} & \cdots & w_{k,m_{k}}\\
	\vline height 1ex & & \vline height 1ex
      \end{bmatrix}\in\mathrm{Mat}_{(r+n)\times m_{i}}(\mathbb{Z})
    \end{equation*}
    where \(e_{k}=(\delta_{k,1},\ldots,\delta_{k,r})^{\intercal}\).
    Let
    \begin{equation*}
      A=
      \begin{bmatrix}
      A_{1} & \cdots & A_{r}
      \end{bmatrix}.
    \end{equation*}
    We also assume that \(A\) has full rank
    and the columns of \(A\) generate \(\mathbb{Z}^{r+n}\) as
    an abelian group.
    The matrix \(A\) is \emph{homogeneous} in the sense of
    Gel'fand--Kapranov--Zelevinskii~\cite{1989-Gelfand-Kapranov-Zelevinski-hypergeometric-functions-and-toral-manifolds}.
  \item Let
    \((\mathbb{C}^{\ast})^{r}\times T=
    \{(s,t)=(s_{1},\ldots,s_{r},t_1,\ldots,t_{n})~|~s_{i},t_{j}
    \in \mathbb{C}^{\ast}\}\) be an algebraic
    torus of dimension \(r+n\).
  \item Let \(\tau_{k}\colon (\mathbb{C}^{\ast})^{r}\times T \to V_{k}\)
  be the morphism defined by \(A_{k}\)
    \begin{equation*}
    \tau_{k}\colon  (s,t) \mapsto (s_{k}\cdot t^{w_{k,1}}, \ldots, s_{k}\cdot t^{w_{k,m_{k}}})
    \end{equation*}
    and \(\bar{\tau}_{k}\) be the composition
    \((\mathbb{C}^{\ast})^{r} \times T \to V_{k}\setminus \{0\}\to \mathbf{P}V_k\).
    Let \(\tau=(\tau_1,\ldots,\tau_r)\).
    and \(\bar{\tau}=(\bar{\tau}_1,\ldots,\bar{\tau}_r)\).
    Notice that \(\tau\) is \emph{injective} under our assumption on \(A\).
   \item Let \(X''\) be the Zariski closure of the image of \(\bar{\tau}\).
    Then \(X'\) is a toric variety (possibly non-normal) with a
    maximal torus \(T'=\operatorname{Im}(\bar{\tau})(\cong T)\).
    Let \(X\to X'\) be any toric resolution.
  \item Let \(\mathcal{L}_{k}^{-1}\) be the pullback
  of \(\mathscr{O}(1)\) on \(\mathbf{P}V_{k}\) along \(X\to X'\).
  Note that the line bundle
  \(\mathcal{L}_{k}^{-1}\) is equipped with a \(T\)-linearization
  such that its divisor polytope \(\Delta_{k}\) is the convex hull of
  \(\{w_{k,1},\ldots,w_{k,m_{k}}\}\) and
  \(V_{k}^{\vee}\subset \mathrm{H}^{0}(X,\mathcal{L}_{k}^{-1})\)
  is a basepoint free linear system generated by \(\{t^{w_{k,j}}~|~1\le j\le m_{k}\}\).
 \end{enumerate}
  Given a parameter \(\beta \in \mathbb{C}^{r+n}\),
  the \(A\)-hypergeometric ideal \(\mathscr{I}_{A,\beta}\)
  is the left ideal of the Weyl algebra \(\mathscr{D}=\mathbb{C}[x,\partial]\) on
  the \emph{dual} vector space \(V^{\vee}\) generated by the following two types of
  operators.

\begin{itemize}
\itemsep=3pt
\item The \emph{box operators}: \(\partial^{\nu_+} - \partial^{\nu_{-}}\),
where \(\nu_{\pm}\in \mathbb{Z}_{\geq 0}^{m}\) satisfy \(A\nu_+=A\nu_{-}\);
\item The \emph{Euler operators}: \(\sum_{j,k} a^{k}_{i,j}
x_{i,j}\partial_{i,j}-\beta_{i}\) for \(i=1,\ldots,r+n\).
\end{itemize}
The \(A\)-hypergeometric system \(\mathcal{M}_{A,\beta}\)
is the cyclic \(\mathscr{D}\)-module
\begin{equation*}
\mathcal{M}_{A,\beta} = \mathscr{D}/\mathscr{I}_{A,\beta}.
\end{equation*}
\end{situation}
In this paper, we will mainly deal with the case when
\begin{equation}
\label{assumption:beta}
\beta  = (\beta_{1},\ldots,\beta_{r},
0,\ldots,0)~\mbox{with}~\beta_{i}\in\mathbb{Q}\setminus\mathbb{Z}.
\end{equation}

\begin{situation}
  \label{situation:hypothesis}
  \textbf{Hypothesis.}~We assume that
  \(-\beta\) lies in the \emph{interior} of the cone generated by
  column vectors of \(A\).
  This implies that \(\beta\) is \emph{semi-nonresonant} (cf.~\cite{2020-Lee-Zhang-a-hypergeometric-systems-and-relative-cohomology}*{\S2}).
\end{situation}

\begin{situation}
\label{sit:toric-main-theorem}
Let \(\Sigma\) be the fan defining \(X\) in item (6).
Denote by \(D_{\rho}\) the toric divisor given by
the one cone \(\rho\in\Sigma(1)\). For each \(k\) we may express the
invertible sheaf \(\mathcal{L}_{k}^{-1}\) as
\begin{equation}
\textstyle
\mathcal{L}_{k}^{-1} = \mathscr{O}_{X}(\sum_{\rho\in \Sigma(1)} a_{\rho,k} D_{\rho}),~
a_{\rho,k}\in\mathbb{Z},
\end{equation}
as a \(T\)-linearized line bundle according to item (7).

For a parameter \(\beta\) and the integers \(a_{\rho,k}\)
as above, put
\begin{equation}
I:=\left\{\rho\in \Sigma(1)~\left|~ \sum_{k=1}^{r} a_{\rho,k}\beta_{k} \in\mathbb{Z}\right.\right\}.
\end{equation}
Let \(J:=\Sigma(1)\setminus I\) be the complement. Finally, let
\(X(I)=X\setminus \cup_{j\in J} D_{j}\).
\end{situation}

We are now ready to state our main result in this paper.
\begin{theorem}
\label{thm:main}
Given an integral matrix \(A\) as in \S\ref{situation:a-hypergeometric-notation}
and a parameter \(\beta\) in
\eqref{assumption:beta}, under the hypothesis \S\ref{situation:hypothesis}
and notation in \S\ref{sit:toric-main-theorem},
we have for \(b\in V^{\vee}\)
\begin{equation}
\label{eq:main-theorem}
\operatorname{Sol}^{0}(\mathcal{M}_{A,\beta})_{b}
\cong \mathrm{H}_{n}(U_{b},U_{b}\cap (\cup_{i\in I} D_{i}),\mathscr{L}_{\beta,b}).
\end{equation}
A few explanations are in order.
\begin{itemize}
\item For \(b=(b_{1},\ldots,b_{r})\), put
\(U_{b}=X(I)\setminus \cup_{i=1}^{r} \{b_{i}=0\}\).
\item \(\mathscr{L}_{\beta,b}\) is the local system on \(U_{b}\) having
monodromy exponent \(\beta_{i}\) around \(\{b_{i}=0\}\).
This can be constructed as follows.
Let \(f_{i}\) be a meromorphic section of \(\mathcal{L}_{i}^{-1}\)
such that \(\operatorname{div}(f_{i})=\sum_{\rho} a_{\rho,i}D_{\rho}\).
Consider the regular functions
\begin{equation*}
r_{i}:=\frac{b_{i}}{f_{i}}
\colon T\rightarrow \mathbb{C}^{\ast}.
\end{equation*}
Denote by \(\mathscr{L}_{i}\) the local system on \(\mathbb{C}^{\ast}\)
having monodromy exponent \(\beta_{i}\) around \(0\). Then their tensor product
\begin{equation*}
\bigotimes_{i=1}^{r} r_{i}^{\ast} \mathscr{L}_{i}
\end{equation*}
is a local system on \(T\). One can check that it can be extended
across \(D_{\rho}\) for \(\rho\in I\). This defines the local system \(\mathscr{L}_{\beta,b}\).

\item \(\mathrm{Sol}^{0}(-)=R^{0}\mathcal{H}om_{\mathscr{D}_{V^{\vee}}^{\mathrm{an}}}
((-)^{\mathrm{an}},\mathscr{O}_{V^{\vee}}^{\mathrm{an}})\)
is the classical solution functor and the subscript \(b\)
denotes the stalk at \(b\in V^{\vee}\).
\end{itemize}
\end{theorem}

We illustrate our main theorem by the following examples.
\begin{instance}
\label{example:1}
Let
\begin{equation*}
A = \begin{bmatrix}
1 & 1 & 1\\
0 & 1 &-1
\end{bmatrix}~\mbox{and}~\beta=
\begin{bmatrix}
-1/2\\
0
\end{bmatrix}.
\end{equation*}
One can easily check that \(A\) and \(\beta\) satisfy the assumptions
and hypothesis we made in both \S\ref{situation:a-hypergeometric-notation}
and Hypothesis \ref{situation:hypothesis}.

In the present case, \(X=\mathbb{P}^{1}\) and \(\mathcal{L}^{-1}=
\mathscr{O}_{X}(D_{0}+D_{\infty})\). Here \(D_{0}\) (resp.~\(D_{\infty}\)) is the Weil
divisor associated with the \(1\)-cone \(\mathbb{R}_{\ge 1}\) (resp.~\(\mathbb{R}_{\le 0}\)).
We have \(I=\emptyset\) and \(X(I)=X\setminus (D_{0}\cup D_{\infty})=\mathbb{C}^{\ast}\)
and therefore for \(b\in V^{\vee}=
\mathrm{H}^{0}(X,\mathscr{O}_{X}(D_{0}+D_{\infty}))\)
\begin{equation}
\operatorname{Sol}^{0}(\mathcal{M}_{A,\beta})_{b}
\cong \mathrm{H}_{1}(U_{b},\mathscr{L}_{\beta,b})
\end{equation}
where \(U_{b}=\mathbb{C}^{\ast}\setminus \{b=0\}\) and
\(\mathscr{L}_{\beta,b}\) is the local system having
monodromy exponent \(1/2\) around \(\{b=0\}\).
Indeed, one can check in this example \(\beta\) is indeed \emph{non-resonant}
in the sense of Gelfand--Kapranov--Zelevinsky
\cite{1990-Gelfand-Kapranov-Zelevinsky-generalized-euler-integrals-and-a-hypergeometric-functions}. The results are consistent.
\end{instance}

\begin{instance}
Let
\begin{equation*}
A = \begin{bmatrix}
1 & 1 & 1 & 1 \\
0 & 1 & 2 &-1
\end{bmatrix}~\mbox{and}~\beta=
\begin{bmatrix}
-1/2\\
0
\end{bmatrix}.
\end{equation*}
One can checks that \(A\) and \(\beta\) satisfy the assumptions
and hypothesis we made in both \S\ref{situation:a-hypergeometric-notation}
and Hypothesis \ref{situation:hypothesis}.
Notice that \(\beta\) is \emph{not} non-resonant in this case.

In the present case, \(X=\mathbb{P}^{1}\) and \(\mathcal{L}^{-1}=
\mathscr{O}_{X}(D_{0}+2D_{\infty})\). (Both \(D_{0}\) and \(D_{\infty}\)
are defined in Example \ref{example:1}.)
We have \(I=\{\rho_{\infty}\}\) and \(X(I)=\mathbb{C}\) and therefore for \(b\in V^{\vee}=
\mathrm{H}^{0}(X,\mathscr{O}_{X}(D_{0}+2D_{\infty}))\)
\begin{equation}
\operatorname{Sol}^{0}(\mathcal{M}_{A,\beta})_{b}
\cong \mathrm{H}_{1}(U_{b},U_{b}\cap(D_{0}\cup D_{\infty}),\mathscr{L}_{\beta,b})
\end{equation}
where \(U_{b}=\mathbb{C}\setminus \{b=0\}\) and
\(\mathscr{L}_{\beta,b}\) is the local system having
monodromy exponent \(1/2\) around \(\{b=0\}\).

For general \(b\in V^{\vee}\), we claim
\begin{equation}
\dim \mathrm{H}_{1}(U_{b},U_{b}\cap(D_{0}\cup D_{\infty}),\mathscr{L}_{\beta,b})=3
\end{equation}
which is equal to the normalized volume of \(A\) as expected.
This can be seen from the long exact sequence of relative homology
\begin{align*}
&\mathrm{H}_{1}(U_{b}\cap(D_{0}\cup D_{\infty}),\mathscr{L}_{\beta,b})\to
\mathrm{H}_{1}(U_{b},\mathscr{L}_{\beta,b})\to
\mathrm{H}_{1}(U_{b},U_{b}\cap(D_{0}\cup D_{\infty}),\mathscr{L}_{\beta,b})\\
&\to \mathrm{H}_{0}(U_{b}\cap(D_{0}\cup D_{\infty}),\mathscr{L}_{\beta,b})\to
\mathrm{H}_{0}(U_{b},\mathscr{L}_{\beta,b})\to
\mathrm{H}_{0}(U_{b},U_{b}\cap(D_{0}\cup D_{\infty}),\mathscr{L}_{\beta,b})\to 0.
\end{align*}
For generic \(b\), the set \(U_{b}\cap (D_{0}\cup D_{\infty})\) consists of one point
which gives the vanishing of the first term and
\begin{equation}
\dim \mathrm{H}_{0}(U_{b}\cap(D_{0}\cup D_{\infty}),\mathscr{L}_{\beta,b})=1.
\end{equation}
Moreover, we have
\begin{equation*}
\dim \mathrm{H}_{1}(U_{b},\mathscr{L}_{\beta,b})=2,~\mbox{and}~
\dim \mathrm{H}_{0}(U_{b},\mathscr{L}_{\beta,b})=0
\end{equation*}
since there is at least one non-integral monodromy exponent in \(\mathscr{L}_{\beta,b}\).
\end{instance}

\begin{instance}
Let
\begin{equation*}
A = \begin{bmatrix}
1 & 1 & 1 & 1 & 1\\
0 & 1 & 2 &-1 &-2
\end{bmatrix}~\mbox{and}~\beta=
\begin{bmatrix}
-1/2\\
0
\end{bmatrix}.
\end{equation*}
One can easily check that \(A\) and \(\beta\) satisfy the assumptions
and hypothesis we made in both \S\ref{situation:a-hypergeometric-notation}
and Hypothesis \ref{situation:hypothesis}. Note that
\(\beta\) is also \emph{not} non-resonant in this case.

In the present case, \(X=\mathbb{P}^{1}\) and \(\mathcal{L}^{-1}=
\mathscr{O}_{X}(2D_{0}+2D_{\infty})\). (Both \(D_{0}\) and \(D_{\infty}\)
are defined in Example \ref{example:1}.)
We have \(I=\{\rho_{0},\rho_{\infty}\}\) and \(X(I)=X\) and therefore for \(b\in V^{\vee}=
\mathrm{H}^{0}(X,\mathscr{O}_{X}(2D_{0}+2D_{\infty}))\)
\begin{equation}
\operatorname{Sol}^{0}(\mathcal{M}_{A,\beta})_{b}
\cong \mathrm{H}_{1}(U_{b},U_{b}\cap(D_{0}\cup D_{\infty}),\mathscr{L}_{\beta,b})
\end{equation}
where \(U_{b}=X\setminus \{b=0\}\) and
\(\mathscr{L}_{\beta,b}\) is the local system having
monodromy exponent \(1/2\) around \(\{b=0\}\).

One can check that for general \(b\in V^{\vee}\),
\begin{equation}
\dim \mathrm{H}_{1}(U_{b},U_{b}\cap(D_{0}\cup D_{\infty}),\mathscr{L}_{\beta,b})=4
\end{equation}
which is equal to the normalized volume of \(A\) as expected.
This can be seen from the long exact sequence of relative homology
\begin{align*}
&\mathrm{H}_{1}(U_{b}\cap(D_{0}\cup D_{\infty}),\mathscr{L}_{\beta,b})\to
\mathrm{H}_{1}(U_{b},\mathscr{L}_{\beta,b})\to
\mathrm{H}_{1}(U_{b},U_{b}\cap(D_{0}\cup D_{\infty}),\mathscr{L}_{\beta,b})\\
&\to \mathrm{H}_{0}(U_{b}\cap(D_{0}\cup D_{\infty}),\mathscr{L}_{\beta,b})\to
\mathrm{H}_{0}(U_{b},\mathscr{L}_{\beta,b})\to
\mathrm{H}_{0}(U_{b},U_{b}\cap(D_{0}\cup D_{\infty}),\mathscr{L}_{\beta,b})\to 0.
\end{align*}
For generic \(b\), the set \(U_{b}\cap (D_{0}\cup D_{\infty})\) consists of two points
which gives the vanishing of the first term and
\begin{equation}
\dim \mathrm{H}_{0}(U_{b}\cap(D_{0}\cup D_{\infty}),\mathscr{L}_{\beta,b})=2.
\end{equation}
Moreover, we have
\begin{equation*}
\dim \mathrm{H}_{1}(U_{b},\mathscr{L}_{\beta,b})=2,~\mbox{and}~
\dim \mathrm{H}_{0}(U_{b},\mathscr{L}_{\beta,b})=0
\end{equation*}
since there is at least one non-integral monodromy exponent in \(\mathscr{L}_{\beta,b}\).
\end{instance}

\section{Generalities on algebraic \texorpdfstring{\(\mathscr{D}\)}{D}-modules}
\label{sec:d-mod}
In this section, we recall some basic
notions in algebraic \(\mathscr{D}\)-modules.
  Let \(X\) be a \emph{smooth} algebraic variety and \(\mathscr{D}_X\) be the sheaf of
  algebraic differential operators on \(X\). By a \(\mathscr{D}_X\)-module on \(X\) we
  always mean a \emph{left} \(\mathscr{D}_X\)-module. Let
  \(\mathrm{D}^{b}_{\mathrm{h}}(\mathscr{D}_X)\) be the bounded derived category of
  \(\mathscr{D}\)-modules over \(X\) with \emph{holonomic} cohomology sheaves.
  Let \(\mathrm{D}^{b}_{\mathrm{rh}}(\mathscr{D}_X)\) be the derived category of
  \(\mathscr{D}_X\)-modules with \emph{regular holonomic} cohomology sheaves. One can
  define the \emph{duality functor}, denoted by
  \(\mathcal{M}\mapsto\mathbb{D}\mathcal{M}\), on
  \(\mathrm{D}_{\mathrm{h}}^b(\mathscr{D}_X)\).
  Let \(f\colon X\to Y\) be a morphism between smooth
  varieties. One can define the following functors
\begin{itemize}
\item For a complex \(\mathcal{M}\in \mathrm{D}_{\mathrm{h}}^b(\mathscr{D}_X)\), let
  \(f_+(\mathcal{M}):=Rf_\ast(\mathscr{D}_{Y\leftarrow X}\otimes^{\mathrm{L}}_{\mathscr{D}_{X}}
  \mathcal{M})\), where
  \(\mathscr{D}_{Y\leftarrow X}\) is the transfer \((f^{-1}\mathscr{D}_{Y},\mathscr{D}_{X})\)-bimodule. 
\item For a complex \(\mathcal{N}\in \mathrm{D}_{\mathrm{h}}^b(\mathscr{D}_Y)\), let
  \(f^!\mathcal{N}:=f^\ast\mathcal{N}[\dim X-\dim Y]\), where \(f^\ast\) is the
  derived pullback on the category of quasi-coherent \(\mathscr{O}_Y\)-modules.
\end{itemize}
Note that these functors can be defined on the category of
\(\mathscr{D}\)-modules without the holonomic condition. Nonetheless, all the
functors \(\mathbb{D}\), \(f_+\) and \(f^!\) preserve the holonomicity. We put
\begin{itemize}
  \item \(f^+:=\mathbb{D}_X f^!\mathbb{D}_Y\), and
  \item \(f_!:=\mathbb{D}_Y f_+\mathbb{D}_X\).
\end{itemize}
\(f^+\) is the left adjoint of \(f_+\) and \(f_!\) is the left adjoint of \(f^!\).

When \(f\) is a smooth morphism, or more generally non-characteristic with respect
to a holonomic \(\mathscr{D}\)-module \(\mathcal{M}\), we have
\(f^{\ast}\mathcal{M}=f^{!}\mathcal{M}[\dim Y-\dim X]=f^{+}\mathcal{M}[\dim
X-\dim Y]\). Finally, given a cartisian diagram
\begin{equation*}
  \begin{tikzcd}
    X' \ar[r,"g'"] \ar[d,"f'"] & X \ar[d,"f"] \\
    Y' \ar[r,"g"] & Y,
  \end{tikzcd}
\end{equation*}
with all varieties are smooth, then we have the base change formula
\begin{equation*}
  g^{!} f_+ = f'_+{ g' }^!.
\end{equation*}

For a smooth complex algebraic variety \(X\), the functor
\begin{equation}
\mathrm{dR}_{X}^{\mathrm{an}}\colon \mathrm{D}^{b}_{\mathrm{rh}}(\mathscr{D}_{X})
\to \mathrm{D}_{\mathrm{c}}^{b}(X^{\mathrm{an}}),~\mathcal{M}^{\bullet}\mapsto
\left(\omega_{X}\otimes^{\mathrm{L}}_{\mathscr{D}_{X}}\mathcal{M}^{\bullet}\right)^{\mathrm{an}}
\end{equation}
gives an equivalence of categories between the
bounded derived category of regular holonomic \(\mathscr{D}_{X}\)-modules
and the bounded derived category of
algebraically constructible sheaves.

\begin{definition}
Let \(f\colon E\to \mathbb{A}^1_{y}\) be a morphism between
smooth algebraic varieties. We define the \emph{exponential
\(\mathscr{D}\)-module} on \(E\) to be
\begin{equation}
\exp(f):=f^\ast(\mathscr{D}_{\mathbb{A}^1_{y}}/(\partial_y-1))=
f^!(\mathscr{D}_{\mathbb{A}^1}/(\partial_y-1))[1-\dim E].
\end{equation}
This is a holonomic \(\mathscr{D}\)-module on \(E\)
which is however irregular at infinity.
\end{definition}

Let \(S\subset X\) be a (possibly singular) subscheme of \(X\) and \(\mathscr{I}_S\)
be the corresponding ideal sheaf. For a \(\mathscr{O}_X\)-module \(\mathscr{F}\) on
\(X\), we define
\begin{equation*}
  \Gamma_{[S]}(\mathscr{F}):= \varinjlim_k \mathcal{H}om_{\mathscr{O}_X}(\mathscr{O}_X/\mathscr{I}^k_S,\mathscr{F}).
\end{equation*}
The quasi-coherent \(\mathscr{O}_{X}\)-module \(\Gamma_{[S]}(\mathscr{F})\)
inherits a \(\mathscr{D}_X\)-module structure and we can consider its right
derived functor \(R\Gamma_{[S]}\). When \(\mathcal{M}\) is a complex with holonomic
cohomology sheaves, so is \(R\Gamma_{[S]}(\mathcal{M})\). Let \(j\colon X\setminus S\to
X\) be the open embedding. For \(\mathcal{M}\in \mathrm{D}^b_{\mathrm{h}}(\mathscr{D}_X)\) we have
the distinguished triangle
\begin{equation}\label{distinguished:triangle:singular:support}
  R\Gamma_{[S]}(\mathcal{M})\to \mathcal{M} \to j_+j^!\mathcal{M}\to R\Gamma_{[S]}(\mathcal{M})[1].
\end{equation}
Let \(i\colon S \to X\) be the closed embedding. In case \(S\) is \emph{smooth},
we have \(R\Gamma_{[S]}(\mathcal{M})\simeq i_+i^!\mathcal{M}\) and the
distinguished triangle \eqref{distinguished:triangle:singular:support} becomes
\begin{equation}
  i_+i^!\mathcal{M}\to \mathcal{M} \to j_+j^!\mathcal{M}\to.
\end{equation}
Therefore we shall sometimes abuse notation and use \(i_{+}i^!\) instead of
\(R\Gamma_{[S]}\) even when \(S\) is singular.
Proofs of the said results can be found in
\cite{2004-Baldassarri-DAgnolo-on-dwork-cohomology-and-algebraic-d-modules}.

\section{Reductions}
\label{sec:reductions}
In this section, based on Reichelt's result,
we demonstrate how to relate the GKZ system
\(\mathcal{M}_{A}^{\beta}\) with certain
exponentially twisted \(\mathscr{D}\)-module.
We will follow the notation set up in \S\ref{situation:a-hypergeometric-notation}.

\begin{situation}{\bf GKZ systems and Fourier--Laplace transforms}.
\label{sit:reichelt}
Since \(\beta\) is non-semiresonant, according to
\cite{2014-Reichelt-laurent-polynomials-gkz-hypergeometric-systems-and-mixed-hodge-modules}*{Proposition 1.14},
we have
  \begin{equation}
  \label{eq:fourier-gkz}
    \mathrm{FT}(\tau_{!}\mathscr{O}_{(\mathbb{C}^{\ast})^{r}\times T}^{\beta})
    = \mathcal{M}_{A,\beta}.
  \end{equation}
Here \(\mathrm{FT}\) stands for the Fourier--Laplace transform of
\(\mathscr{D}\)-modules and \(\mathscr{O}_{(\mathbb{C}^{\ast})^{r}\times T}^{\beta}\)
is the integrable connection
\begin{equation*}
\mathscr{D}_{(\mathbb{C}^{\ast})^{r}\times T}\slash
\mathscr{D}_{(\mathbb{C}^{\ast})^{r}\times T}
\langle s_{i}\partial_{s_{i}}-\beta_{i},t_{j}
\partial_{t_{j}}~|~1\le i\le r,~1\le j\le m_{i}\rangle.
\end{equation*}
\end{situation}

\begin{situation}{\bf Exponentially twisted de Rham complexes}.
\label{sit:exp-twisted-de-rham}
We explain how the Fourier--Laplace transform in
\eqref{eq:fourier-gkz} is related to
exponentially twisted de Rham complexes.

We have basepoint free line bundles \(\mathcal{L}_{1}^{-1},\ldots,
\mathcal{L}_{r}^{-1}\) and hence surjections
\begin{equation}
\label{eq:surjection}
V_{k}^{\vee}\otimes_{\mathbb{C}}\mathscr{O}_{X}\to \mathcal{L}_{k}^{-1}
\end{equation}
using the basis \(\{t^{w_{k,1}},\ldots,t^{w_{k,m_{k}}}\}\).

Denote by \(\mathbb{L}_{1},\ldots,\mathbb{L}_{r}\)
the geometric line bundles associated to \(\mathcal{L}_{1},\ldots,\mathcal{L}_{r}\).
As an algebraic variety,
\begin{equation}
\mathbb{L}_{k} = \operatorname{Spec}_{\mathscr{O}_{X}}
\operatorname{Sym}_{\mathscr{O}_{X}}^{\bullet}\mathcal{L}_{k}^{-1}.
\end{equation}
Taking \(\operatorname{Spec}_{\mathscr{O}_{X}}(-)\) on
\eqref{eq:surjection} and composing with
the projection \(V_{k}\times X\to V_{k}\), we obtain proper morphisms
\begin{equation}
\label{eq:prod-proper}
\mathbb{L}_{k}\to V_{k}.
\end{equation}
Note that \(\mathbb{L}_{k}\) acquires a toric structure
via the distinguished \(T\)-linearization.
According to the construction of \eqref{eq:surjection}, we see that
\begin{equation}
\begin{tikzcd}
&(\mathbb{C}^{\ast})^{r}\times T\ar[r,"\iota"]\ar[rr,bend right,"\tau"]
&\mathbb{L}_{1}\times_{X}\cdots\times_{X}\mathbb{L}_{r}\ar[r,"b"]
&V=V_{1}\times\cdots\times V_{r}
\end{tikzcd}
\end{equation}
is identically equal to \(\tau\). Here
\(\iota\) is the inclusion of the maximal torus in
the toric variety \(\mathbb{L}_{1}\times_{X}\cdots\times_{X}\mathbb{L}_{r}\)
while \(b\) is the product of \eqref{eq:prod-proper}.

We wish to compute \(\mathrm{FT}(\tau_{!}
\mathscr{O}_{(\mathbb{C}^{\ast})^{r}\times T})=
\mathrm{FT}(b_{!}\iota_{!}
\mathscr{O}_{(\mathbb{C}^{\ast})^{r}\times T})\).
Looking at the diagram
\begin{equation*}
\begin{tikzcd}
& & & \mathbb{A}^{1}\\
& (\mathbb{C}^{\ast})^{r}\times T\times V^{\vee}\ar[r,"\iota\times\operatorname{id}"]
\ar[d,"\mathrm{pr}"]
& \mathbb{L}_{1}\times_{X}\cdots \times_{X} \mathbb{L}_{r}\times V^{\vee}\ar[ru,"f"]
\ar[d,"\operatorname{pr}"]\ar[r,"b\times\operatorname{id}"]
& V\times V^{\vee}\ar[d,"\operatorname{pr}_{V}"]\ar[u,"F"']
\ar[r,"\operatorname{pr_{V^{\vee}}}"'] &V^{\vee} &\\
&(\mathbb{C}^{\ast})^{r}\times T\ar[r,"\iota"]
& \mathbb{L}_{1}\times_{X}\cdots \times_{X} \mathbb{L}_{r} \ar[r,"b"] &V, & &
\end{tikzcd}
\end{equation*}
by the properness of \(b\), we have
\begin{align}
\label{eq:exponential-twist}
\mathrm{FT}(b_{!}\iota_{!}
\mathscr{O}_{(\mathbb{C}^{\ast})^{r}\times T})&=
\mathrm{FT}(b_{+}\iota_{!}
\mathscr{O}_{(\mathbb{C}^{\ast})^{r}\times T})\notag\\
&=\operatorname{pr}_{V^{\vee}+}(\operatorname{pr}_{V}^{\ast}
b_{+}\iota_{!}
\mathscr{O}_{(\mathbb{C}^{\ast})^{r}\times T}\otimes \exp(F))\notag\\
&=\operatorname{pr}_{V^{\vee}+}((b\times\operatorname{id})_{+}
\operatorname{pr}^{\ast}\iota_{!}
\mathscr{O}_{(\mathbb{C}^{\ast})^{r}\times T}\otimes \exp(F))\notag\\
&=\operatorname{pr}_{V^{\vee}+}((\iota\times\operatorname{id})_{!}
\mathscr{O}_{(\mathbb{C}^{\ast})^{r}\times T\times V^{\vee}}\otimes \exp(f)).
\end{align}
Here \(F\) is the canonical pairing.
The last equality holds by base change formula
since \(\operatorname{pr}^{\ast}=\operatorname{pr}^{+}[\dim V^{\vee}] =
\operatorname{pr}^{!}[-\dim V^{\vee}]\).

We can further decompose the projection \(
\mathbb{L}_{1}\times_{X}\cdots\times_{X}\mathbb{L}_{r}\times V^{\vee}\to V^{\vee}\) into
\begin{equation*}
\mathbb{L}_{1}\times_{X}\cdots\times_{X}\mathbb{L}_{r}\times V^{\vee}
\xrightarrow{\pi\times\mathrm{id}}
X\times V^{\vee} \xrightarrow{\mathrm{pr}_{V^{\vee}}} V^{\vee}
\end{equation*}
and therefore \eqref{eq:exponential-twist} becomes
\begin{equation}
\label{eq:exp-twist-comp}
\operatorname{pr}_{V^{\vee}+}
(\pi\times\mathrm{id})_{+}((\iota\times\operatorname{id})_{!}
\mathscr{O}_{(\mathbb{C}^{\ast})^{r}\times T\times V^{\vee}}\otimes \exp(f)).
\end{equation}
\end{situation}
We can summarize the result in the following proposition.
\begin{proposition}
\label{prop:reduction}
We have an isomorphism
\begin{equation}
\mathcal{M}_{A,\beta}\cong \operatorname{pr}_{V^{\vee}+}
(\pi\times\mathrm{id})_{+}((\iota\times\operatorname{id})_{!}
\mathscr{O}^{\beta}_{(\mathbb{C}^{\ast})^{r}\times T\times V^{\vee}}\otimes \exp(f)).
\end{equation}
\end{proposition}

\section{A local computation on direct images of
\texorpdfstring{\(\mathscr{D}\)}{D}-modules}
\label{sec:regularity}
\begin{situation}
\label{sit:direct-image-local-line}
Consider the following situation.
\begin{itemize}
\item Let \(X\) be an affine smooth algebraic variety.
\item Let \(\pi\colon X\times\mathbb{C}\to X\) be a trivial line bundle
and let \(s\) be the coordinate on \(\mathbb{C}\).
\item Let \(j\colon X\times\mathbb{C}^{\ast}\to X\times \mathbb{C}\)
be the open inclusion.
\item Let \(g\) be a regular function on \(X\).
\end{itemize}
For a parameter \(\beta\in\mathbb{C}\), we
consider the integrable connection
\(\mathscr{O}_{\mathbb{C}^{\ast}}^{\beta}\). To be precise,
\(\mathscr{O}_{\mathbb{C}^{\ast}}^{\beta}\cong\mathscr{O}_{\mathbb{C}^{\ast}}\)
as a coherent \(\mathscr{O}_{\mathbb{C}^{\ast}}\)-module and
the \(\partial_{s}\) action is given by
\begin{equation}
\partial_{s} \star s^{n} := (\beta+n) s^{n-1}.
\end{equation}
We can think of it as a connection by ``twisting \(s^{\beta}\),'' i.e.,
\(\partial_{s}\star s^{n} = s^{-\beta}\partial_{s} s^{\beta}\cdot s^{n}\).

The fiber diagram (denoting the pullback morphisms by the same notation)
\begin{equation}
\begin{tikzcd}
& &X &\\
& X\times\mathbb{C}^{\ast}\ar[r,"j"]\ar[d,"\mathrm{pr}"] &X\times\mathbb{C}
\ar[d,"\mathrm{pr}"]\ar[r,"f=sg"]\ar[u,"\pi"] &\mathbb{C}\\
& \mathbb{C}^{\ast}\ar[r,"j"] &\mathbb{C} &
\end{tikzcd}
\end{equation}
implies that
\(j_{!}\mathrm{pr}^{\ast}\mathscr{O}_{\mathbb{C}^{\ast}}^{\beta}=
\mathrm{pr}^{\ast}j_{!}\mathscr{O}_{\mathbb{C}^{\ast}}^{\beta}\).
\end{situation}

We are interested in
\(\pi_{+}(\mathrm{pr}^{\ast}j_{!}\mathscr{O}_{\mathbb{C}^{\ast}}^{\beta}\otimes\exp(f))\)
for non-integral \(\beta\).
To facilitate the computation,
let \(p\in X\) and denote \(\mathscr{O}=\mathscr{O}_{X,p}\)
and \(\mathscr{K}=\mathscr{O}[g^{-1}]\).
Choose a local coordinate system \(x=(x_{1},\ldots,x_{n})\).
Note that for \(\beta\notin\mathbb{Z}\), we have
\begin{equation}
j_{!}\mathscr{O}_{\mathbb{C}^{\ast}}^{\beta}=j_{+}\mathscr{O}_{\mathbb{C}^{\ast}}^{\beta}=
\mathbb{C}[s,\partial_{s}]\slash (s\partial_{s}-\beta+1).
\end{equation}
Then locally at
\(p\), \(\pi_{+}(\mathrm{pr}^{\ast}j_{!}\mathscr{O}_{\mathbb{C}^{\ast}}^{\beta}\otimes\exp(f))\)
is represented by the complex
\begin{equation}
\begin{tikzcd}
&\mathscr{O}[s,\partial_{s}]\slash (s\partial_{s}-\beta+1)
\ar[r,"D_{f}"]
&\mathscr{O}[s,\partial_{s}]\slash (s\partial_{s}-\beta+1),
\end{tikzcd}
\end{equation}
where \(D_{f}(u)=\partial_{s}u+(\partial f\slash\partial s)u=
\partial_{s}u + gu\). Owing to the exponential twist,
the \(\mathscr{D}_{X}\)-module structure on
\(\mathscr{O}[s,\partial_{s}]\slash (s\partial_{s}-\beta+1)\) is given by
the action
\begin{equation}
\partial_{x_{i}}\cdot u = \frac{\partial u}{\partial x_{i}} + s
\frac{\partial g}{\partial x_{i}} u.
\end{equation}

We need a few lemmas.
\begin{lemma}
If \(\beta\notin\mathbb{Z}\), the morphism
\begin{equation}
\begin{tikzcd}
&\mathscr{O}[s,\partial_{s}]\slash (s\partial_{s}-\beta+1)
\ar[r,"D_{f}"]
&\mathscr{O}[s,\partial_{s}]\slash (s\partial_{s}-\beta+1),
\end{tikzcd}
\end{equation}
is injective.
\end{lemma}
\begin{proof}
\(\operatorname{Ker}(D_{f})\) consists of exponential functions, which are
obviously non-algebraic. Hence \(D_{f}\) is injective.
\end{proof}

\begin{situation}
Let \(\mathscr{B}=\mathscr{K}\slash\mathscr{O}\).
Consider the following commutative ladder
\begin{equation*}
\begin{tikzcd}[column sep=1.5em]
&\mathscr{O}[s,\partial_{s}]\slash (s\partial_{s}-\beta+1)\ar[hookrightarrow,r]\ar[d,"D_{\mathscr{O}}=\partial_{s}+g"]
&\mathscr{K}[s,\partial_{s}]\slash (s\partial_{s}-\beta+1)\ar[twoheadrightarrow,r]\ar[d,"D_{\mathscr{K}}=\partial_{s}+g"]
&\mathscr{B}[s,\partial_{s}]\slash (s\partial_{s}-\beta+1)\ar[d,"D_{\mathscr{B}}=\partial_{s}+g"]\\
&\mathscr{O}[s,\partial_{s}]\slash (s\partial_{s}-\beta+1)\ar[hookrightarrow,r]
&\mathscr{K}[s,\partial_{s}]\slash (s\partial_{s}-\beta+1)\ar[twoheadrightarrow,r]
&\mathscr{B}[s,\partial_{s}]\slash (s\partial_{s}-\beta+1)
\end{tikzcd}
\end{equation*}
with exact rows and \(D_{\mathscr{O}}=D_{f}\) in the above diagram.
The same reason implies that both \(D_{\mathscr{K}}\) and
\(D_{\mathscr{B}}\) are injective.
The snake lemma gives the short exact sequence
\begin{equation}
\begin{tikzcd}
&0\ar[r]
&\operatorname{Coker}(D_{\mathscr{O}})\ar[r]
&\operatorname{Coker}(D_{\mathscr{K}})\ar[r]
&\operatorname{Coker}(D_{\mathscr{B}})\ar[r]
&0.
\end{tikzcd}
\end{equation}
\end{situation}

Next, we examine the cokernel of \(D_{\mathscr{K}}\).
\begin{lemma}
If \(\beta\notin\mathbb{Z}\), \(\operatorname{Coker}(D_{\mathscr{K}})\cong\mathscr{K}\)
as \(\mathscr{K}\)-modules.
\end{lemma}
\begin{proof}
Note that any element \(r\in \mathscr{K}[s,\partial_{s}]\slash (s\partial_{s}-\beta+1)\)
can be written as
\begin{equation}
r \equiv \sum_{i=0}^{p} c_{i}(x)\partial_{s}^{i}+\sum_{j=1}^{q} d_{j}(x) s^{j}
\mod{s\partial_{s}-\beta+1}.
\end{equation}
Such an \(r\) belongs to the image of \(D_{\mathscr{K}}\) if
the set of equations
\begin{equation}
\label{eq:set-equations-beta}
\begin{cases}
a_{i}(x) + g(x) a_{i+1}(x) &= c_{i+1}(x),~i=0,1,2,\ldots\\
\beta b_{1}(x) + g(x) a_{0}(x) &= c_{0}(x),\\
(j+\beta)b_{j+1}(x) + g(x) b_{j}(x) &= d_{j}(x),~j=1,2,\ldots
\end{cases}
\end{equation}
has solutions \(a_{i}(x),b_{j}(x)\in\mathscr{K}\)
such that \(a_{N}(x)=0\) for \(N\gg 0\) and
\(b_{M}(x)=0\) for \(M\gg 0\).
To solve the equation, we start with \(c_{p}(x)\).
We can set \(a_{N}(x)=0\) for \(N>p\) and
\(a_{p-1}(x) = c_{p}(x)\). Utilizing the first equation, we can solve
\begin{equation}
\label{eq:a-beta}
a_{k}(x) = c_{k+1}(x) - g(x)c_{k+2}(x) + \cdots +
(-g(x))^{p-k-1} c_{p}(x),~k=0,1,2,\ldots.
\end{equation}
In particular,
\begin{equation}
a_{0}(x) = c_{1}(x) - g(x)c_{2}(x) + \cdots + (-1)^{p-1}g(x)^{p-1} c_{p}(x).
\end{equation}
Now let us look at the third equation. We again set \(b_{M}(x)=0\) for \(M>q\)
and \(g(x)b_{q}(x) = d_{q}(x)\). Since \(g\) is a unit in \(\mathscr{K}\), we can solve
\begin{equation}
b_{q}(x) = \frac{d_{q}(x)}{g(x)}.
\end{equation}
By a backward induction, using the formula
\begin{equation}
b_{j}(x) = \frac{-(j+\beta) b_{j+1}(x) +d_{j}(x)}{g(x)}=
-\frac{(j+\beta)b_{j+1}(x)}{g(x)} + \frac{d_{j}(x)}{g(x)},
\end{equation}
we can recursively solve
\begin{equation}
\label{eq:b-beta}
b_{j}(x) = \frac{\sum_{k=j}^{q} (-1)^{k-j}\Gamma(k+\beta)
d_{k}(x) g(x)^{q-k}}{\Gamma(j+\beta) g(x)^{q-j+1}}.
\end{equation}
In particular,
\begin{equation}
b_{1}(x) = \frac{\sum_{k=1}^{q} (-1)^{k-1}
\Gamma(k+\beta) d_{k}(x) g(x)^{q-k}}{\Gamma(1+\beta) g(x)^{q}}.
\end{equation}
Finally \(a_{0}(x)\) and \(b_{1}(x)\) have to obey the second equation, i.e.,
\begin{equation*}
c_{0}(x) = \frac{\beta\sum_{k=1}^{q} (-1)^{k-1}
\Gamma(k+\beta) d_{k}(x) g(x)^{q-k}}{\Gamma(1+\beta) g(x)^{q}} +
g(x)\sum_{l=1}^{p} (-1)^{l-1}g(x)^{l-1}c_{l}(x).
\end{equation*}
This is equivalent to saying that
\begin{align*}
\frac{\beta\sum_{k=1}^{q}(-1)^{k-1}
\Gamma(k+\beta) d_{k}(x) g(x)^{q-k}}{\Gamma(1+\beta) g(x)^{q}} =
\sum_{l=0}^{p} (-1)^{l}g(x)^{l}c_{l}(x).
\end{align*}
That is, the coefficients satisfy a linear equation over \(\mathscr{K}\):
\begin{equation}
-\beta\sum_{k=1}^{q} \Gamma(k+\beta)
(-g(x))^{-k} d_{k}(x) - \Gamma(1+\beta)\sum_{l=0}^{p} (-g(x))^{l}c_{l}(x)=0.
\end{equation}
Conversely, given any coefficients \(c_{i}(x)\) and \(d_{j}(x)\) satisfying the equation above,
we can always solve for \(a_{i}(x)\) and \(b_{j}(x)\) using the formulae
\eqref{eq:a-beta} and \eqref{eq:b-beta}.
This shows that \(\operatorname{Coker}(D_{2})\cong\mathscr{K}\).
\end{proof}

\begin{situation}
The computation above motivates the following definition.
For a non-zero element \(g\in\mathscr{O}\),
we define a linear functional \(\mathfrak{L}^{\beta}\colon
\mathscr{K}[s,\partial_{s}]\slash\langle s\partial_{s}-\beta+1\rangle\to \mathscr{K}\)
via
\begin{equation*}
\sum_{i=0}^{\infty} c_{i}(x)\partial_{s}^{i} +
\sum_{j=1}^{\infty} d_{j}(x)s^{j} \mapsto
\beta\sum_{k=1}^{\infty} \Gamma(k+\beta)
(-g(x))^{-k} d_{k}(x) + \Gamma(1+\beta)\sum_{l=0}^{\infty} (-g(x))^{l}c_{l}(x).
\end{equation*}
Note that this is well-defined, since \(c_{i}(x)=d_{j}(x)=0\) for \(i,j\) sufficiently large.
From the calculation, we see that \(\operatorname{Ker}(\mathfrak{L}^{\beta}) =
\operatorname{Im}(D_{\mathscr{K}})\) and
\begin{equation}
\mathfrak{L}^{\beta}\colon \operatorname{Coker}(D_{\mathscr{K}}) \to \mathscr{K}
\end{equation}
is an isomorphism.
\end{situation}

From the proof, we see that what we
need is the invertibility of \(g\). Thus, the same
calculation implies the following corollary.
\begin{corollary}
If \(g(p)\ne 0\), i.e., \(g\) is a unit in \(\mathscr{O}\), we have
\begin{equation}
\operatorname{Coker}(D_{\mathscr{O}})\cong \mathscr{O}
\end{equation}
as \(\mathscr{O}\)-modules. Moreover, the isomorphism
is induced by \(\mathfrak{L}^{\beta}\).
\end{corollary}

\begin{lemma}
If \(p\) is a zero of \(g(x)\), then
the map \(D_{\mathscr{B}}\) is surjective, i.e., \(\operatorname{Coker}(D_{\mathscr{B}})=0\).
Consequently, we have \(\operatorname{Coker}(D_{\mathscr{O}})\cong
\operatorname{Coker}(D_{\mathscr{K}})\).
\end{lemma}
\begin{proof}
Again any element \(r\in \mathscr{B}[s,\partial_{s}]\slash (s\partial_{s}-\beta+1)\)
can be written as
\begin{equation}
r \equiv \sum_{i=0}^{p} c_{i}(x)\partial_{s}^{i}+\sum_{j=1}^{q} d_{j}(x) s^{j}
\mod{s\partial_{s}-\beta+1}
\end{equation}
with \(c_{i}(x),d_{j}(x)\in\mathscr{B}\).
Such an \(r\) belongs to the image of \(D_{\mathscr{B}}\) if
there are \(a_{i}(x),b_{j}(x)\in\mathscr{B}\)
with \(a_{i}(x)=b_{j}(x)=0\) for sufficiently large \(i,j\)
solving the set of equations
\begin{equation}
\label{eq:set-equations-beta-hyperfunctions}
\begin{cases}
a_{i}(x) + g(x) a_{i+1}(x) &= c_{i+1}(x),~i=0,1,2,\ldots\\
\beta b_{1}(x) + g(x) a_{0}(x) &= c_{0}(x),\\
(j+\beta)b_{j+1}(x) + g(x) b_{j}(x) &= d_{j}(x),~j=1,2,\ldots
\end{cases}
\end{equation}
Again we begin with \(a_{i}(x)\). We can set \(a_{N}(x)=0\) for \(N>p\) and
\(a_{p-1}(x) = c_{p}(x)\). Utilizing the first equation, we can solve
\begin{equation*}
a_{k}(x) = c_{k+1}(x) - g(x)c_{k+2}(x) + \cdots + (-1)^{p-k-1}
g(x)^{p-k-1} c_{p}(x),~k=0,1,2,\ldots.
\end{equation*}
In particular,
\begin{equation*}
a_{0}(x) = c_{1}(x) - g(x)c_{2}(x) + \cdots + (-1)^{p-1}g(x)^{p-1} c_{p}(x).
\end{equation*}
Now we use the second equation to solve \(b_{1}(x)\). We get
\begin{equation}
b_{1}(x)=\frac{c_{0}(x)-g(x)a_{0}(x)}{\beta}=\beta^{-1}\sum_{k=0}^{p} (-g(x))^{k}c_{k}(x).
\end{equation}
We can continue solving \(b_{j+1}(x)\) using the third equation:
\begin{equation}
b_{j+1}(x) = \frac{d_{j}(x)-g(x)b_{j}(x)}{j+\beta}.
\end{equation}
We can easily solve
\begin{equation}
b_{j+1}(x)=\frac{(-g(x))^{j}\Gamma(1+\beta)b_{1}(x)+\sum_{k=1}^{j} \Gamma(k+\beta)(-g(x))^{j-k}
d_{k}(x)}{\Gamma(j+1+\beta)}.
\end{equation}
Since \(d_{j}(x)=0\) for all \(j\) sufficient large, we see that
\begin{equation}
b_{M+1}(x) = g(x)^{M-q}\cdot F(x)
\end{equation}
where \(F(x)\in\mathscr{B}\) is independent of \(M\) as long as \(M>q\).
Therefore, because \(g(x)\) has a zero at \(p\), we see that
\(b_{M}(x)=0\) for all \(M\) sufficiently large.
\end{proof}

\begin{situation}
Now we study the \(\mathscr{D}_{X}\)-module structure on
\(\operatorname{Coker}(D_{\mathscr{O}})\).
If \(g(p)\ne 0\), we have \(\mathfrak{L}^{\beta}(1)=\Gamma(1+\beta)\)
and
\begin{equation}
\mathfrak{L}^{\beta}(\partial_{x_{i}}\cdot 1)
=\mathfrak{L}^{\beta}(s\partial_{x_{i}}g(x))=
-\frac{\beta\Gamma(1+\beta)(\partial g/\partial x_{i})}{g(x)}.
\end{equation}
We conclude that the induced
\(\mathscr{D}_{X}\)-module structure
on \(\operatorname{Coker}(D_{\mathscr{O}})\cong \mathscr{O}\)
is given by the integrable connection
\begin{equation}
u\mapsto \mathrm{d} u + \beta \cdot \frac{\mathrm{d}g(x)}{g(x)}\wedge u.
\end{equation}
If \(p\) is a zero of \(g(x)\), the same computation shows that
the induced
\(\mathscr{D}_{X}\)-module structure
on \(\operatorname{Coker}(D_{\mathscr{O}})\cong \mathscr{K}\)
is given by the same formula (regarded as a meromorphic connection with a
pole along \(\{g=0\}\))
\begin{equation}
u\mapsto \mathrm{d} u + \beta \cdot \frac{\mathrm{d}g(x)}{g(x)}\wedge u.
\end{equation}
The argument still goes through when \(p\) is a pole of \(g\), except we should now
consider directly the complex
\begin{equation}
\begin{tikzcd}
&\mathscr{K}[s,\partial_{s}]\slash (s\partial_{s}-\beta+1)\ar[r,"D_{\mathscr{K}}"]
&\mathscr{K}[s,\partial_{s}]\slash (s\partial_{s}-\beta+1)
\end{tikzcd}
\end{equation}
instead of \(\mathscr{O}[s,\partial_{s}]\slash (s\partial_{s}-\beta+1)\),
and there will be no \(\mathscr{B}\) factor.
\end{situation}

\begin{theorem}
Under the Situation \ref{sit:direct-image-local-line},
the \(\mathscr{D}_{X}\)-module
\(\pi_{+}(\mathrm{pr}^{\ast}j_{!}\mathscr{O}_{\mathbb{C}^{\ast}}^{\beta}\otimes
\exp(f))\) is regular holonomic on \(X\).
Moreover, if \(g\) is smooth,
then its (covariant) Riemann--Hilbert partner is \(R\rho_{\ast}\mathscr{L}_{\beta}^{\vee}\)
where \(\rho\colon U\to X\) is the open embedding
of the complement \(U=X\setminus \{g=0\}\),
\(\mathscr{L}_{\beta}\) is the local system on \(U\)
whose monodromy around \(\{g=0\}\) has exponent \(\beta\),
and \(\mathscr{L}_{\beta}^{\vee}\) is its dual.
\end{theorem}
\begin{proof}
For \(\beta\in\mathbb{Z}\),
the \(\mathscr{D}_{X}\)-module in question
is nothing but \(\pi_{+}(\mathrm{pr}^{\ast} j_{!}\mathscr{O}_{\mathbb{C}^{\ast}}
\otimes\exp(f))\). It is sitting in the distinguished triangle
\begin{equation*}
\pi_{+}(\mathrm{pr}^{\ast} j_{!}\mathscr{O}_{\mathbb{C}^{\ast}}
\otimes\exp(f))\to
\pi_{+}(\mathrm{pr}^{\ast} \mathscr{O}_{\mathbb{C}}
\otimes\exp(f))\to
\pi_{+}(\mathrm{pr}^{\ast} i_{+}i^{+}\mathscr{O}_{\mathbb{C}}
\otimes\exp(f)).
\end{equation*}
The third term is isomorphic to \(\mathscr{O}_{X}\)
by projection formula whereas
the second term is isomorphic to \(\alpha_{+}\alpha^{!}\mathscr{O}_{X}\)
where \(\alpha\colon \{g=0\}\to X\) is the closed embedding.
(This follows from a comparison theorem between relative de Rham cohomology and Dwork cohomology
due to Dimca \textit{et al}. See also
\cite{2000-Dimca-Maaref-Sabbah-Saito-dwork-cohomology-and-algebraic-d-modules}
or \cite{2004-Baldassarri-DAgnolo-on-dwork-cohomology-and-algebraic-d-modules}.)
The result then follows.

For \(\beta\notin\mathbb{Z}\), this follows from
the discussion above.
\end{proof}

We can generalize the results to
the case of vector bundles.
\begin{situation}
\label{sit:pushforward-vector-bundle-local}
Consider the following situation.
\begin{itemize}
\item Let \(X\) be an affine smooth algebraic
variety.
\item Let \(\pi\colon X\times\mathbb{C}^{r}\to X\) be
a trivial vector bundle of rank \(r\) and
\((s_{1},\ldots,s_{r})\) be a coordinate system on \(\mathbb{C}^{r}\).
\item Let \(j\colon X\times(\mathbb{C}^{\ast})^{r}\to X\times\mathbb{C}^{r}\)
be the inclusion.
\item Let \(g_{1},\ldots,g_{r}\) be
regular functions on \(X\).
\item Let \(\beta=(\beta_{1},\ldots,\beta_{r})\in (\mathbb{C}\setminus\mathbb{Z})^{r}\).
\end{itemize}
Consider the integrable connection
\(\mathscr{O}_{\mathbb{C}^{\ast}}^{\beta_{k}}\) on
\(\mathbb{C}\) with coordinate \(s_{k}\)
and set
\begin{equation}
\mathscr{O}_{(\mathbb{C}^{\ast})^{r}}^{\beta}=
\mathscr{O}_{\mathbb{C}^{\ast}}^{\beta_{1}}\boxtimes\cdots\boxtimes
\mathscr{O}_{\mathbb{C}^{\ast}}^{\beta_{r}}.
\end{equation}
The fiber diagram (again denoting the pullback morphisms by the same notation)
\begin{equation}
\begin{tikzcd}[column sep=5em]
& &X &\\
& X\times(\mathbb{C}^{\ast})^{r}\ar[r,"j"]\ar[d,"\mathrm{pr}"] &X\times\mathbb{C}^{r}
\ar[d,"\mathrm{pr}"]\ar[r,"f=\sum_{k=1}^{r}s_{k}g_{k}"]\ar[u,"\pi"] &\mathbb{C}\\
& (\mathbb{C}^{\ast})^{r}\ar[r,"j"] &\mathbb{C}^{r} &
\end{tikzcd}
\end{equation}
implies that
\(j_{!}\mathrm{pr}^{\ast}\mathscr{O}_{(\mathbb{C}^{\ast})^{r}}^{\beta}=
\mathrm{pr}^{\ast}j_{!}\mathscr{O}_{(\mathbb{C}^{\ast})^{r}}^{\beta}\).
In the present situation, we wish to compute the \(\mathscr{D}_{X}\)-module
\(\pi_{+}(j_{!}\mathrm{pr}^{\ast}
\mathscr{O}_{(\mathbb{C}^{\ast})^{r}}^{\beta}\otimes\exp(f))\).

Let \(p\in X\). Put \(\mathscr{O}=\mathscr{O}_{X,p}\) and
\(\mathscr{K}=\mathscr{O}[g_{1}^{-1}\cdots g_{r}^{-1}]\).
Consider for each \(1\le k\le r\) a complex of \(\mathscr{D}_{X}\)-modules
\begin{equation*}
  \mathcal{M}_{k} = \left[
  \mathscr{O}[s_{k},\partial_{s_{k}}]\slash (s_{k}\partial_{s_{k}}-\beta_{k}+1)
  \xrightarrow{D_{k}}\mathscr{O}[s_{k},\partial_{s_{k}}]
  \slash (s_{k}\partial_{s_{k}}-\beta_{k}+1)\right],
\end{equation*}
where \(D_{k}(u)=\partial_{s_{k}}u-g_{k}(x)u\) as before.
Their (derived) tensor product
\begin{equation*}
\mathcal{M}_{1}\otimes_{\mathscr{O}_{X}}^{\mathrm{L}}\cdots
\otimes_{\mathscr{O}_{X}}^{\mathrm{L}}\mathcal{M}_{r}
\end{equation*}
represents \(\pi_{+}(j_{!}\mathrm{pr}^{\ast}
\mathscr{O}_{(\mathbb{C}^{\ast})^{r}}^{\beta}\otimes\exp(f))\).
\end{situation}

The following theorem is immediate.
\begin{theorem}
\label{thm:monodormy-base}
In Situation \ref{sit:pushforward-vector-bundle-local},
the \(\mathscr{D}_{X}\)-module
\(\pi_{+}(j_{!}\mathrm{pr}^{\ast}
\mathscr{O}_{(\mathbb{C}^{\ast})^{r}}^{\beta}\otimes\exp(f))\)
is regular holonomic on \(X\).
Moreover, if \(\{g_{1}\cdots g_{r}=0\}\) is a simple normal crossing divisor,
then its (covariant) Riemann--Hilbert partner is \(R\rho_{\ast}\mathscr{L}_{\beta}^{\vee}\)
where \(\rho\colon U\to X\) is the open embedding
of the complement \(U=X\setminus \{g_{1}\cdots g_{r}=0\}\),
\(\mathscr{L}_{\beta}\) is the local system on \(U\)
whose monodromy around \(\{g_{k}=0\}\) has exponent \(\beta_{k}\),
and \(\mathscr{L}_{\beta}^{\vee}\) is its dual.
\end{theorem}


\section{The \texorpdfstring{\(!\)}{!}-pushforward}
\label{sec:push}
Let us resume the notation in \S\ref{sec:reductions}.
In this section, we will analyze the \(\mathscr{D}\)-module
\((\iota\times\operatorname{id})_{!}
\mathscr{O}_{(\mathbb{C}^{\ast})^{r}\times T\times V^{\vee}}\).

\begin{situation}
\label{sit:bundle-space-connection}
Let us consider the following situation.
\begin{itemize}
\item Let \(X\) be a smooth toric variety of dimension \(n\) defined by a fan \(\Sigma\).
\item Let \(T\subset X\) be the maximal torus with coordinates \((t_{1},\ldots,t_{n})\).
\item Let \(D_{\rho}\) denote the Weil divisor associated with
the one cone \(\rho\in\Sigma(1)\).
\item Let \(\mathcal{L}_{1}^{-1},\ldots,\mathcal{L}_{r}^{-1}\) be
invertible sheaves on \(X\). For each \(k\), there
are integers \(a_{\rho,k}\) (indexed by \(\rho\in\Sigma(1)\))
such that \(\mathcal{L}_{k}^{-1}\cong\mathscr{O}_{X}(\sum a_{\rho,k}D_{\rho})\).
The integers \(a_{\rho,k}\) are not unique. But
we will fix once for all a choice and hence an isomorphism
\begin{equation*}
\textstyle\mathcal{L}_{k}^{-1}\cong\mathscr{O}_{X}(\sum a_{\rho,k}D_{\rho}),~k=1,\ldots,r.
\end{equation*}
\item Let \(\mathbb{L}_{1},\ldots,\mathbb{L}_{r}\) be geometric line bundles
associated with the dual \(\mathcal{L}_{1},\ldots,\mathcal{L}_{r}\).
Explicitly, we have
\begin{equation*}
\mathbb{L}_{k} = \operatorname{Spec}_{\mathscr{O}_{X}}
\operatorname{Sym}_{\mathscr{O}_{X}}^{\bullet}\mathcal{L}_{k}^{-1}.
\end{equation*}
\end{itemize}

Under the identification \(\mathcal{L}_{k}^{-1}\cong
\mathscr{O}_{X}(\sum a_{\rho,k}D_{\rho})\),
the vector bundle \(\mathbb{L}_{1}\times_{X}\cdots\times_{X}\mathbb{L}_{r}\)
acquires a canonical toric structure which we now describe.

We need some terminology in toric geometry.
Let \(N=\mathbb{Z}^{n}\) be the lattice such that
its scalar extension \(N_{\mathbb{R}}=N\otimes_{\mathbb{Z}}\mathbb{R}\) is the
Euclidean space where \(\Sigma\) sits.
Denote by the same notation \(\rho\) the primitive generator
of the \(1\)-cone \(\rho\in\Sigma(1)\). To each \(\rho\in\Sigma(1)\) we
associate an integral vector
\begin{equation}
\tilde{\rho}:=(\rho,a_{\rho,1},\ldots,a_{\rho,r})\in N\times\mathbb{Z}^{r}.
\end{equation}
Now for any \(\sigma\in\Sigma(n)\) we define a \((r+n)\)-dimensional cone
\begin{equation}
\tilde{\sigma}:=\operatorname{Cone}(\{\tilde{\rho}~|~\rho\in\sigma(1)\}\cup
\{(\mathbf{0},e_{1}),\ldots,(\mathbf{0},e_{r})\}).
\end{equation}
Let \(\Theta\) be the collection of all \(\tilde{\sigma}\) and all
their faces. One can verify that \(\Theta\) is a fan defining
the toric variety \(\mathbb{L}_{1}\times_{X}\cdots\times_{X}\mathbb{L}_{r}\).
\end{situation}

\begin{situation}
Now assume that \(\mathrm{H}^{0}(X,\mathcal{L}_{k}^{-1})\ne 0\) for all \(k\).
It is known that
\begin{equation}
\mathrm{H}^{0}(X,\mathcal{L}_{k}^{-1})\cong \bigoplus_{m\in\Delta_{k}\cap M}
\mathbb{C}t^{m}
\end{equation}
where \(M=\mathrm{Hom}_{\mathbb{Z}}(N,\mathbb{Z})\) is the dual lattice and
\(\Delta_{k}\) is the polyhedron associated to \(\mathcal{L}_{k}^{-1}
\cong\mathscr{O}_{X}(\sum a_{\rho,k}D_{\rho})\)
\begin{equation*}
\Delta_{k}=\{m\in M_{\mathbb{R}}~|~\langle \rho,m\rangle\ge -a_{\rho,k},~
\forall\rho\in\Sigma(1)\}.
\end{equation*}
Note that a choice of the identification \(\mathcal{L}_{k}^{-1}
\cong\mathscr{O}_{X}(\sum a_{\rho,k}D_{\rho})\) uniquely determines
the polyhedron \(\Delta_{k}\).

Now assume that \(V_{k}^{\vee}\subset\mathrm{H}^{0}(X,\mathcal{L}_{k}^{-1})\)
is a subspace spanned by \(\{t^{w_{k,1}},\ldots,t^{w_{k,m_{k}}}\}
\subset \Delta_{k}\cap M\).
The set gives rise to a morphism
\(V_{k}^{\vee}\otimes_{\mathbb{C}}\mathscr{O}_{X}
\to\mathcal{L}_{k}^{-1}\)
and therefore a morphism of algebraic varieties
\(\mathbb{L}_{k}\to V_{k}\times X\).
Combined with the projection to \(V_{k}\), we obtain \(\mathbb{L}_{k}\to V_{k}\).
Consider their product
\begin{equation}
\mathbb{L}_{1}\times_{X}\cdots\times_{X}\mathbb{L}_{r} \to V:=V_{1}\times\cdots\times V_{r}.
\end{equation}
\end{situation}

We can readily check the lemma.
\begin{lemma}
\label{lem:integral-points}
Let notation be as above.
Then the composition
\begin{equation}
(\mathbb{C}^{\ast})^{r}\times T\to
\mathbb{L}_{1}\times_{X}\cdots\times_{X}\mathbb{L}_{r} \to V
\end{equation}
is given by an integral matrix \(A =
\begin{bmatrix}
A_{1} & \cdots & A_{r}
\end{bmatrix}\) with
\begin{equation}
A_{k} = \begin{bmatrix}
e_{k} & \cdots & e_{k}\\
\vline height 1ex & & \vline height 1ex\\
w_{k,1} & \cdots & w_{k,m_{k}}\\
\vline height 1ex & & \vline height 1ex
\end{bmatrix}\in\mathrm{Mat}_{(r+n)\times m_{k}}(\mathbb{Z}).
\end{equation}
\end{lemma}
\begin{proof}
\end{proof}

\begin{situation}
\label{sit:open-inclusion}
Let \(\iota\colon (\mathbb{C}^{\ast})^{r}\times T\to
\mathbb{L}_{1}\times_{X}\cdots\times_{X}\mathbb{L}_{r}\)
be the open embedding in Lemma \ref{lem:integral-points}.
As we have seen, we have to compute
the \(!\)-pushforward
\begin{equation*}
\iota_{!}\mathscr{O}_{(\mathbb{C}^{\ast})^{r}\times T}^{\beta}
\end{equation*}
where \(\mathscr{O}_{(\mathbb{C}^{\ast})^{r}\times T}^{\beta}\)
is the pullback of \(\mathscr{O}_{(\mathbb{C}^{\ast})^{r}}^{\beta}\),
the integrable connection defined in Situation
\ref{sit:pushforward-vector-bundle-local}, along
the projection \((\mathbb{C}^{\ast})^{r}\times T \to (\mathbb{C}^{\ast})^{r}\).
To this end,
we take an affine open cover \(U_{\tilde{\sigma}}\)
with \(\sigma\in\Sigma(n)\). Recall that
\begin{equation*}
\tilde{\sigma}=\operatorname{Cone}(\{\tilde{\rho}~|~\rho\in\sigma(1)\}\cup
\{(\mathbf{0},e_{1}),\ldots,(\mathbf{0},e_{r})\}).
\end{equation*}
\end{situation}

\begin{lemma}
\label{lem:coor-local}
The affine toric variety associated with \(\tilde{\sigma}\) is
the spectrum of the ring
\begin{equation*}
\mathbb{C}[\tilde{\sigma}^{\vee}\cap (M\times\mathbb{Z}^{r})]
= \mathbb{C}[t^{\nu_{\tau}},t^{\nu_{k}}s_{k}~|~\tau\in\sigma(1),~k=1,\ldots,r]
\end{equation*}
where \(\nu_{\tau}\in M\) is the element defining the facet \(\sigma(1)\setminus \{\tau\}\)
and \(\nu_{k}\in M\) is the Cartier data of \(\mathcal{L}_{k}^{-1}\) on
\(U_{\sigma}\).
Consequently, under the trivialization
\begin{equation}
\begin{tikzcd}
&\left.\mathbb{L}_{1}\times_{X}\cdots\times_{X}\mathbb{L}_{r}\right|_{U_{\sigma}}\ar[rd]
\ar[r,"\cong"]&U_{\tilde{\sigma}}\cong\mathbb{C}^{n}\times\mathbb{C}^{r}\ar[d]\\
& & U_{\sigma}\cong\mathbb{C}^{n}
\end{tikzcd}
\end{equation}
\(t^{\nu_{\tau}}\) corresponds to the coordinate on \(U_{\sigma}\)
and \(t^{\nu_{1}}s_{1},\ldots,t^{\nu_{r}}s_{r}\) correspond
to the coordinate on the fiber of \(\mathbb{L}_{1},\ldots,\mathbb{L}_{r}\).
\end{lemma}
\begin{proof}
Since \(\tilde{\sigma}\) is smooth of dimension \((r+n)\), its facets
are in one-to-one correspondence with
elements in \(\tilde{\sigma}(1)\). Each facet is defined by a linear
functional, and hence an element in \(M\times\mathbb{Z}^{r}\). The
collection of these elements is the generating set of \(\tilde{\sigma}^{\vee}\).
There are two cases.
\begin{itemize}
\item[(a)] The facet contains \(\tilde{\sigma}(1)\setminus \{e_{k}\}\).
\item[(b)] The facet contains \(\tilde{\sigma}(1)\setminus \{\tilde{\tau}\}\)
for some \(\tau\in \sigma(1)\).
\end{itemize}
Let us denote element in \(M\times \mathbb{Z}^{r}\) by
\((\nu,b_{1},\ldots,b_{r})\).
We deal with the case (a) first. Suppose \(e_{k}\)
is omitted. Then the equations
\begin{equation*}
\langle (\mathbf{0},e_{i}),(\nu,b_{1},\ldots,b_{r})\rangle=0,~i\ne k,
\end{equation*}
imply that \(b_{i}=0\) for \(i\ne k\).
Here \(\langle -,-\rangle\) is the canonical dual pairing
between \(N\times\mathbb{Z}^{r}\) and \(M\times\mathbb{Z}^{r}\).
Also we have \(b_{k}>0\).
For \(\rho\in\sigma(1)\) we have
\begin{equation*}
\langle (\rho,a_{\rho,1},\ldots,a_{\rho,r}),(\nu,b_{1},\ldots,b_{r})\rangle=
\langle \rho,\nu\rangle + a_{\rho,k}b_{k}=0.
\end{equation*}
Since \(\sigma\) is smooth, i.e.~the primitive generators
of \(\sigma(1)\) form a \(\mathbb{Z}\)-basis of \(N\),
we can solve for \(\nu_{k}\in M\)
from the equations.
In this case, we can further assume \(b_{k}=1\).

In case (b), we have
\begin{equation*}
\langle (\mathbf{0},e_{i}),(\nu,b_{1},\ldots,b_{r})\rangle=0,~i=1,\ldots,r,
\end{equation*}
which imply that \(b_{i}=0\) for all \(i=1,\ldots,r\).
Suppose \(\tau\in\sigma(1)\)
is omitted. We have
\begin{equation*}
\langle (\rho,a_{\rho,1},\ldots,a_{\rho,r}),(\nu,0,\ldots,0)\rangle=
\langle \rho,\nu\rangle=0,~\mbox{for}~\rho\ne\tau
\end{equation*}
from which one infers that \(\nu=\nu_{\tau}\in M\) is
the element which defines the facet of \(\sigma\)
associated with \(\sigma(1)\setminus\{\tau\}\).
\end{proof}

We rewrite the connection \(\mathscr{O}_{(\mathbb{C}^{\ast})^{r}\times T}^{\beta}\) on
\(U_{\tilde{\sigma}}\) in terms of the coordinates provided in
Lemma \ref{lem:coor-local}. 
For \(\tau\in\sigma(1)\), we put \(x_{\tau}:=t^{\nu_{\tau}}\)
and \(y_{k}=t^{\nu_{k}}s_{k}\), \(k=1,\ldots,r\).
It is also clear that
\begin{equation*}
\nu_{k} = \sum_{\tau\in\sigma(1)} -a_{\tau,k} \nu_{\tau}.
\end{equation*}
Therefore, we can solve \(t\) and \(s\) in terms of \(x_{\tau}\) and \(y_{k}\)
\begin{align*}
t_{i} &= h_{i}(x_{\tau})\\
s_{k} &= y_{k} \prod_{\tau\in\sigma(1)} x_{\tau}^{a_{\tau,k}}.
\end{align*}

\begin{proposition}
\label{prop:monodromy}
There are complex numbers \(c_{1},\ldots,c_{n}\) such that
\begin{align*}
x_{\tau}\partial_{x_{\tau}} &= \sum_{i=1}^{n} c_{i} t_{i}\partial_{t_{i}} +
\sum_{j=1}^{k} a_{\tau,j} s_{j}\partial_{s_{j}},~\tau\in\sigma(1)\\
y_{k} \partial_{y_{k}} &=s_{k}\partial_{s_{k}},~k=1,\ldots,r.
\end{align*}
Consequently, the integrable connection
\(\mathscr{O}_{(\mathbb{C}^{\ast})^{r}\times T}^{\beta}\) on
\(U_{\tilde{\sigma}}\) has monodromy whose
exponent is equal to \(\sum_{j=1}^{k} a_{\tau,j}\beta_{j}\) around \(x_{\tau}=0\)
and \(\beta_{k}\) around \(y_{k}=0\).
\end{proposition}

\begin{corollary}
\label{cor:connection-one}
Assume the Situation \ref{sit:bundle-space-connection}.
Let \(\beta=(\beta_{1},\ldots,\beta_{r})\in\mathbb{C}^{r}\) and
put
\begin{itemize}
\item \(I\subset\Sigma(1)\) be the subset consisting of \(\rho\in\Sigma(1)\)
such that \(\sum_{j=1}^{k} a_{\rho,j}\beta_{j}\in\mathbb{Z}\). We
can also think of \(I\) as a subset in \(\Theta(1)\) via \(\rho\mapsto\tilde{\rho}\);
\item \(J:=\Theta(1)\setminus I\) be the complement.
\end{itemize}
(Recall that \(\Theta\) is the fan defining \(\mathbb{L}_{1}\times_{X}\cdots
\times_{X}\mathbb{L}_{r}\).)
Then the integrable connection
\(\mathscr{O}_{(\mathbb{C}^{\ast})^{r}\times T}^{\beta}\)
defined in Situation \ref{sit:open-inclusion}
can be extended to an integrable connection across
\begin{equation*}
\cup_{\rho\in I}D_{\rho}\setminus \cup_{\tau\in J} D_{\tau}.
\end{equation*}
\end{corollary}
\begin{proof}
According to Proposition \ref{prop:monodromy},
the integrable connection \(\mathscr{O}_{(\mathbb{C}^{\ast})^{r}\times T}^{\beta}\)
has no monodromy around \(D_{\rho}\) for \(\rho\in I\).
\end{proof}

Let \(V=V_{1}\times\cdots\times V_{r}\) as before.
We can apply the construction and results to
the ``universal connection,'' i.e.~the pullback connection
\(\mathscr{O}_{(\mathbb{C}^{\ast})^{r}\times T\times V^{\vee}}^{\beta}\)
of \(\mathscr{O}_{(\mathbb{C}^{\ast})^{r}\times T}^{\beta}\) along
the projection \((\mathbb{C}^{\ast})^{r}\times T\times V^{\vee}\to
(\mathbb{C}^{\ast})^{r}\times T\).

\begin{corollary}
\label{cor:connection-universal}
Let notation be as in Corollary \ref{cor:connection-one}.
Then the integrable connection
\(\mathscr{O}_{(\mathbb{C}^{\ast})^{r}\times T\times V^{\vee}}^{\beta}\)
can be extended to an integrable connection across
\begin{equation*}
\cup_{\rho\in I}(D_{\rho}\times V^{\vee})\setminus
\cup_{\tau\in J}(D_{\tau}\times V^{\vee}).
\end{equation*}
\end{corollary}

\section{Proof of Theorem \ref{thm:main}}
\label{sec:proof}
Given a matrix \(A\) as in Situation
\ref{situation:a-hypergeometric-notation}
and a parameter \(\beta\) as in \eqref{assumption:beta},
under the hypothesis \S\ref{situation:hypothesis}, by Proposition
\ref{prop:reduction} in Section \ref{sec:reductions}, we have
\begin{equation}
\label{eq:d-module-proof}
\mathcal{M}_{A,\beta}\cong \operatorname{pr}_{V^{\vee}+}
(\pi\times\mathrm{id})_{+}((\iota\times\operatorname{id})_{!}
\mathscr{O}^{\beta}_{(\mathbb{C}^{\ast})^{r}\times T\times V^{\vee}}\otimes \exp(f)).
\end{equation}
Here, we recall that the maps are defined through
the commutative diagram (cf.~\S\ref{sec:reductions})
\begin{equation*}
\begin{tikzcd}
& & & \mathbb{A}^{1}\\
& (\mathbb{C}^{\ast})^{r}\times T\times V^{\vee}\ar[r,"\iota\times\operatorname{id}"]
\ar[d,"\mathrm{pr}"]
& \mathbb{L}_{1}\times_{X}\cdots \times_{X} \mathbb{L}_{r}\times V^{\vee}\ar[ru,"f"]
\ar[d,"\operatorname{pr}"]\ar[r,"b\times\operatorname{id}"]
& V\times V^{\vee}\ar[d,"\operatorname{pr}_{V}"]\ar[u,"F"']
\ar[r,"\operatorname{pr_{V^{\vee}}}"'] &V^{\vee} &\\
&(\mathbb{C}^{\ast})^{r}\times T\ar[r,"\iota"]
& \mathbb{L}_{1}\times_{X}\cdots \times_{X} \mathbb{L}_{r} \ar[r,"b"] &V. & &
\end{tikzcd}
\end{equation*}
Let us concentrate on the \(\mathscr{D}\)-module
\(\mathscr{O}^{\beta}_{(\mathbb{C}^{\ast})^{r}\times T\times V^{\vee}}\).
By Corollary \ref{cor:connection-universal}, we can extend
\(\mathscr{O}^{\beta}_{(\mathbb{C}^{\ast})^{r}\times T\times V^{\vee}}\) to
an integrable connection
on
\begin{equation*}
(\mathbb{C}^{\ast})^{r}\times X(I)\times V^{\vee},\quad X(I):=X\setminus
\cup_{\tau\notin I}D_{\tau},
\end{equation*}
where the subset \(I\subset\Sigma(1)\) is defined in Corollary \ref{cor:connection-one}.

Assume that \(|I|=p\). We introduce the following notation.
\begin{itemize}
\item Let \(D_{1},\ldots,D_{p}\) be the toric divisors associated with elements in \(I\).
\item For each subset \(K\subset\{1,\ldots,p\}\), let \(D_{K}:=\cap_{k\in K} D_{k}\).
\item Let \(\mathbb{E}:=\mathbb{L}_{1}\times_{X}\cdots \times_{X} \mathbb{L}_{r}\)
and \(\mathbb{E}(I):=\left.(\mathbb{L}_{1}\times_{X}\cdots \times_{X} \mathbb{L}_{r})
\right|_{X(I)}\).
\item Denote by \(i_{q}\) the proper map
\begin{equation*}
i_{q}\colon \coprod_{|K|=q} (\mathbb{C}^{\ast})^{r}\times (D_{K}\cap X(I))
\times V^{\vee}\to (\mathbb{C}^{\ast})^{r}\times X(I)\times V^{\vee}
\end{equation*}
induced by the inclusion \(D_{K}\cap X(I)\to X(I)\).
\item Denote by \(\mathscr{O}^{\beta}_{(\mathbb{C}^{\ast})^{r}\times X(I)\times V^{\vee}}\)
the integrable connection on \((\mathbb{C}^{\ast})^{r}\times X(I)\times V^{\vee}\) extended from
\(\mathscr{O}^{\beta}_{(\mathbb{C}^{\ast})^{r}\times T\times V^{\vee}}\)
by Corollary \ref{cor:connection-universal}.
\end{itemize}
We can decompose \(\iota\times\mathrm{id}\) into three open inclusions
\begin{equation*}
\begin{tikzcd}
&(\mathbb{C}^{\ast})^{r}\times T\times V^{\vee}\ar[r,"\gamma"]
\ar[rrr,bend right,"\iota\times\mathrm{id}"]
&(\mathbb{C}^{\ast})^{r}\times X(I)\times V^{\vee}\ar[r,"j"]
&\mathbb{E}(I)\times V^{\vee}\ar[r,"\theta"]
&\mathbb{E}\times V^{\vee}.
\end{tikzcd}
\end{equation*}
We have a triangle
\begin{align}
\label{eq:resolution-of-shriek}
\begin{split}
\gamma_{!}\mathscr{O}^{\beta}_{(\mathbb{C}^{\ast})^{r}\times T\times V^{\vee}}
&\to \mathscr{O}^{\beta}_{(\mathbb{C}^{\ast})^{r}\times X(I)\times V^{\vee}}\\
&\to \left[i_{1+}i_{1}^{+}
\mathscr{O}^{\beta}_{(\mathbb{C}^{\ast})^{r}\times X(I)\times V^{\vee}}
\to\cdots\to i_{p+}i_{p}^{+}
\mathscr{O}^{\beta}_{(\mathbb{C}^{\ast})^{r}\times X(I)\times V^{\vee}}\right].
\end{split}
\end{align}
The complex in the bracket above is induced by the Mayer--Vietoris resolution
of the simple normal crossing divisor \(\cup_{i\in I} (\mathbb{C}^{\ast})^{r}\times
D_{i}\times V^{\vee}\)
restricting to \((\mathbb{C}^{\ast})^{r}\times X(I)\times V^{\vee}\).

The \(\mathscr{D}\)-module in \eqref{eq:d-module-proof} is transformed into
\begin{align}
\operatorname{pr}_{V^{\vee}+}
&(\pi\times\mathrm{id})_{+}((\iota\times\operatorname{id})_{!}
\mathscr{O}^{\beta}_{(\mathbb{C}^{\ast})^{r}\times T\times V^{\vee}}\otimes \exp(f))\notag\\
&=\operatorname{pr}_{V^{\vee}+}(\pi\times\mathrm{id})_{+}(\theta_{!}j_{!}\gamma_{!}
\mathscr{O}^{\beta}_{(\mathbb{C}^{\ast})^{r}\times T\times V^{\vee}}\otimes \exp(f)).
\end{align}
Replacing \(\gamma_{!}
\mathscr{O}^{\beta}_{(\mathbb{C}^{\ast})^{r}\times T\times V^{\vee}}\)
with the complex
\begin{equation}
\mathscr{O}^{\beta}_{(\mathbb{C}^{\ast})^{r}\times X(I)\times V^{\vee}}
\to \left[i_{1+}i_{1}^{+}
\mathscr{O}^{\beta}_{(\mathbb{C}^{\ast})^{r}\times X(I)\times V^{\vee}}
\to\cdots\to i_{p+}i_{p}^{+}
\mathscr{O}^{\beta}_{(\mathbb{C}^{\ast})^{r}\times X(I)\times V^{\vee}}\right]
\end{equation}
in \eqref{eq:resolution-of-shriek}, it is the sufficient to compute the complexes
given by
\begin{align}
\label{eq:2nd-shriek-pushforward}
\begin{split}
\operatorname{pr}_{V^{\vee}+}
&(\pi\times\mathrm{id})_{+}(\theta_{!}j_{!}
\mathscr{O}^{\beta}_{(\mathbb{C}^{\ast})^{r}\times X(I)\times V^{\vee}}\otimes \exp(f))\\
\operatorname{pr}_{V^{\vee}+}
&(\pi\times\mathrm{id})_{+}(\theta_{!}j_{!}i_{q+}i_{q}^{+}
\mathscr{O}^{\beta}_{(\mathbb{C}^{\ast})^{r}\times X(I)\times V^{\vee}}\otimes \exp(f)).
\end{split}
\end{align}
Notice that we have
\begin{equation}
\theta_{!}j_{!}\mathscr{O}^{\beta}_{(\mathbb{C}^{\ast})^{r}\times X(I)\times V^{\vee}}
= \theta_{+}j_{+}\mathscr{O}^{\beta}_{(\mathbb{C}^{\ast})^{r}\times X(I)\times V^{\vee}}
\end{equation}
due to the non-integrality of the monodromy and that
\begin{align*}
\theta_{!}j_{!}i_{q+}i_{q}^{+}
\mathscr{O}^{\beta}_{(\mathbb{C}^{\ast})^{r}\times X(I)\times V^{\vee}} &= \theta_{!}j_{!}
i_{q!}i_{q}^{+}\mathscr{O}^{\beta}_{(\mathbb{C}^{\ast})^{r}\times X(I)\times V^{\vee}}\\
&=i_{q!}\theta_{!}j_{!}i_{q}^{+}
\mathscr{O}^{\beta}_{(\mathbb{C}^{\ast})^{r}\times X(I)\times V^{\vee}}\\
&=i_{q+}\theta_{+}j_{+}i_{q}^{+}
\mathscr{O}^{\beta}_{(\mathbb{C}^{\ast})^{r}\times X(I)\times V^{\vee}}
\end{align*}
Here we have used the commutative diagram
\begin{equation*}
\begin{tikzcd}
&(\mathbb{C}^{\ast})^{r}\times (D_{K}\cap X(I))\times V^{\vee}\ar[r,"j"]\ar[d,"i_{q}"]
&\left.\mathbb{E}(I)\right|_{D_{K}}\times V^{\vee}\ar[r,"\theta"]\ar[d,"i_{q}"]
&\left.\mathbb{E}\right|_{D_{K}}\times V^{\vee}\ar[d,"i_{q}"]\\
&(\mathbb{C}^{\ast})^{r}\times X(I)\times V^{\vee}\ar[r,"j"]
&\mathbb{E}(I)\times V^{\vee}\ar[r,"\theta"]
&\mathbb{E}\times V^{\vee}
\end{tikzcd}
\end{equation*}
and the fact that both \(
\mathscr{O}^{\beta}_{(\mathbb{C}^{\ast})^{r}\times X(I)\times V^{\vee}}\)
and \(i_{q}^{+}\mathscr{O}^{\beta}_{(\mathbb{C}^{\ast})^{r}\times X(I)\times V^{\vee}}\)
have monodromies whose exponents are all non-integral (and therefore \(j_{+}=j_{!}\)
on those sheaves).

Thus, combined with the projection formula,
the equations in \eqref{eq:2nd-shriek-pushforward} become
\begin{align}
\label{eq:2nd-shriek-pushforward-reduced}
\begin{split}
&\operatorname{pr}_{V^{\vee}+}(\pi\times\mathrm{id})_{+}\theta_{+}(
j_{+}\mathscr{O}^{\beta}_{(\mathbb{C}^{\ast})^{r}\times X(I)\times V^{\vee}}
\otimes \exp(f\circ\theta))\\
&\operatorname{pr}_{V^{\vee}+}
(\pi\times\mathrm{id})_{+}i_{q+}\theta_{+}(j_{+}i_{q}^{+}
\mathscr{O}^{\beta}_{(\mathbb{C}^{\ast})^{r}\times X(I)\times V^{\vee}}
\otimes \exp(f\circ \theta\circ i_{q})).
\end{split}
\end{align}
Finally, using the commutative diagrams
\begin{equation*}
\begin{tikzcd}
&\mathbb{E}(I)\times V^{\vee}\ar[r,"\theta"]\ar[d,"\pi\times\mathrm{id}"]
&\mathbb{E}\times V^{\vee}\ar[d,"\pi\times\mathrm{id}"]
&\left.\mathbb{E}(I)\right|_{D_{K}}\times V^{\vee}\ar[r,"\theta"]\ar[d,"\pi\times\mathrm{id}"]
&\left.\mathbb{E}\right|_{D_{K}}\times V^{\vee}\ar[d,"\pi\times\mathrm{id}"]\\
&X(I)\times V^{\vee}\ar[r,"\theta"] &X\times V^{\vee}
&(D_{K}\cap X(I))\times V^{\vee}\ar[r,"\theta"] &D_{K}\times V^{\vee}
\end{tikzcd}
\end{equation*}
and the fact that they are compatible with \(i_{q}\), we see that
\eqref{eq:2nd-shriek-pushforward-reduced} becomes
\begin{align}
\label{eq:3nd-shriek-pushforward-reduced}
\begin{split}
&\operatorname{pr}_{V^{\vee}+}\theta_{+}(\pi\times\mathrm{id})_{+}(
j_{+}\mathscr{O}^{\beta}_{(\mathbb{C}^{\ast})^{r}\times X(I)\times V^{\vee}}
\otimes \exp(f\circ\theta))\\
&\operatorname{pr}_{V^{\vee}+}
i_{q+}\theta_{+}(\pi\times\mathrm{id})_{+}(j_{+}i_{q}^{+}
\mathscr{O}^{\beta}_{(\mathbb{C}^{\ast})^{r}\times X(I)\times V^{\vee}}
\otimes \exp(f\circ \theta\circ i_{q})).
\end{split}
\end{align}
Note that
\begin{equation}
i_{q}^{+}
\mathscr{O}^{\beta}_{(\mathbb{C}^{\ast})^{r}\times X(I)\times V^{\vee}}=
\bigoplus_{|K|=q}
\mathscr{O}^{\beta}_{(\mathbb{C}^{\ast})^{r}\times (D_{K}\cap X(I))\times V^{\vee}}
\end{equation}
is a direct sum of integrable connections of the same type.

Now we are in the position to apply Theorem \ref{thm:monodormy-base}
to equations in \eqref{eq:3nd-shriek-pushforward-reduced}.

The covariant Riemann--Hilbert (RH)
partner of
\begin{equation}
(\pi\times\mathrm{id})_{+}(
j_{+}\mathscr{O}^{\beta}_{(\mathbb{C}^{\ast})^{r}\times X(I)\times V^{\vee}}
\otimes \exp(f\circ\theta))
\end{equation}
is equal to
\(R\rho_{\ast}\mathscr{L}_{\beta}^{\vee}\).

Here \(\mathscr{L}_{\beta}\) is the local system on
\(U:=(X(I)\times V^{\vee})\setminus \cup_{i=1}^{r}\{g_{i}=0\}\)
having monodromy around each \(\{g_{k}=0\}\)
whose exponent is \(\beta_{k}\), \(g_{k}\in V_{k}\otimes V_{k}^{\vee}\)
is the universal section and \(\rho\colon U\to X(I)\times V^{\vee}\)
is the open inclusion.

Similarly, the RH
partner of \((\pi\times\mathrm{id})_{+}(j_{+}i_{q}^{+}
\mathscr{O}^{\beta}_{(\mathbb{C}^{\ast})^{r}\times X(I)\times V^{\vee}}
\otimes \exp(f\circ \theta\circ i_{q}))\) is
\begin{equation}
R\rho_{\ast}i_{q\ast}i_{q}^{-1}\mathscr{L}_{\beta}^{\vee}.
\end{equation}
It follows that the RH partner of
\((\pi\times\mathrm{id})_{+}(\theta_{!}j_{!}\gamma_{!}
\mathscr{O}^{\beta}_{(\mathbb{C}^{\ast})^{r}\times T\times V^{\vee}}\otimes \exp(f))\)
is quasi-isomorphic to the complex
\begin{equation}
R\theta_{\ast}
R\rho_{\ast}\left(\mathscr{L}_{\beta}^{\vee}\to\left[i_{1\ast}i_{1}^{-1}
\mathscr{L}_{\beta}^{\vee}
\to\cdots \to i_{p\ast}i_{p}^{-1}\mathscr{L}_{\beta}^{\vee}\right]\right)
\end{equation}
and the morphisms appearing above are all induced from restrictions.
Then
\begin{equation}
\mathrm{dR}_{V^{\vee}}^{\mathrm{an}}(\mathcal{M}_{A}^{\beta})\underset{qis}{\simeq}
R\mathrm{pr}_{V^{\vee}\ast}R\theta_{\ast}
R\rho_{\ast}\left(
\mathscr{L}_{\beta}^{\vee}\to\left[i_{1\ast}i_{1}^{-1}\mathscr{L}_{\beta}^{\vee}
\to\cdots \to i_{p\ast}i_{p}^{-1}\mathscr{L}_{\beta}^{\vee}\right]\right).
\end{equation}
For \(b\in V^{\vee}\), applying the Verdier duality to the complex above
and taking its stalk at \(b\), we obtain
\begin{equation}
\operatorname{Sol}^{0}(\mathcal{M}_{A,\beta})_{b}
\cong \mathrm{H}_{n}(U_{b},U_{b}\cap(\cup_{i\in I} D_{i}),\mathscr{L}_{\beta,b})
\end{equation}
where \(U_{b}=X(I)\setminus \cup_{i=1}^{r}\{g_{i,b}=0\}\) and \(g_{i,b}\)
is the restriction of the universal section \(g_{i}\in V_{i}\otimes V_{i}^{\vee}\) to \(b\).
This completes the proof.

\end{document}